\documentclass[12pt,a4paper]{article}
 
\usepackage[latin1]{inputenc}
\usepackage[english]{babel}
\usepackage[T1]{fontenc}
\usepackage{amsmath} 
\usepackage{amsfonts}
\usepackage{amssymb}
\usepackage{amsthm}
\usepackage{dsfont}
\usepackage{stmaryrd}
\usepackage{tikz}
\usepackage{todonotes}
\usetikzlibrary{matrix,arrows.meta}
\usepackage[title,titletoc,toc]{appendix}

\allowdisplaybreaks

\usepackage{hyperref}
\usepackage{xcolor}
\usepackage[a4paper,top=2.1cm,bottom=2.60cm,left=2.1cm,right=2.1cm]{geometry} 

\theoremstyle{plain}
\newtheorem{theorem}{Theorem}[section]
\newtheorem{lemma}[theorem]{Lemma}
\newtheorem{proposition}[theorem]{Proposition}

\theoremstyle{definition} 
\newtheorem{remark}[theorem]{Remark}

\begin{document}

\author{Roberto Bramati \thanks{Dipartimento di Ingegneria Gestionale, dell'Informazione e della Produzione, Universit\`a degli Studi di Bergamo, Viale Marconi 5,
24044 Dalmine (BG), Italy. Email: roberto.bramati@unibg.it} \thanks{Department of Mathematics: Analysis, Logic and Discrete Mathematics, Ghent University, Krijgslaan 281, 9000 Ghent, Belgium.  Email: roberto.bramati@ugent.be} , Matteo Dalla Riva\thanks{Dipartimento di Ingegneria, Universit\`a degli Studi di Palermo, Viale delle Scienze, Ed. 8, 90128 Palermo, Italy. Email: matteo.dallariva@unipa.it} ,  Paolo Luzzini\thanks{Dipartimento di Scienze e Innovazione Tecnologica, Universit\`a degli Studi del Piemonte Orientale ``Amedeo Avogadro'', Viale Teresa Michel 11, 15121 Alessandria, Italy. Email: paolo.luzzini@uniupo.it} ,  Paolo Musolino\thanks{Dipartimento di Matematica ``Tullio Levi-Civita'', Universit\`a degli Studi di Padova, Via Trieste 63, 35121 Padova, Italy. Email: paolo.musolino@unipd.it}}

\title{Periodic layer potentials and {domain perturbations}}

\date{December 17, 2024}


\maketitle

\noindent
{\bf Abstract:} In this paper, we review the construction of periodic fundamental solutions and periodic layer potentials for various differential operators. Specifically, we focus on the Laplace equation, the Helmholtz equation, the Lam\'e system, and the heat equation. We then describe how these layer potentials can be applied to analyze domain perturbation problems. In particular, we present applications to the asymptotic behavior of quasi-periodic solutions for a Dirichlet problem for the Helmholtz equation in an unbounded domain with small periodic perforations as the size of each hole tends to $0$. Additionally, we investigate the dependence of spatially periodic solutions of an initial value Dirichlet problem for the heat equation on regular perturbations of the base of a parabolic cylinder.

\vspace{9pt}

\vspace{9pt}


\section{Introduction}

Potential theory is a powerful tool for studying boundary value problems. By employing integral representations involving single-layer and double-layer potentials, it enables the transformation of boundary value problems in an open set $\Omega$ into integral equations defined on its boundary, $\partial \Omega$. These integral equations can then be analyzed using Fredholm theory to determine their solvability. Then, coming back to the integral representations, one can construct solutions to the original boundary value problems.

On the other hand, concrete applications often lead to boundary value problems in periodic domains,  as in the works of  Ammari, Kang, and Lim \cite{AmKaLi06}, Ammari, Kang, and Touibi \cite{AmKaTo05},  Dryga\'s, Gluzman, Mityushev, and Nawalaniec \cite{DrGlMiNa20}, Gluzman, Mityushev, and Nawalaniec \cite{GlMiNa18}, and Kapanadze, Mishuris, and Pesetskaya \cite{KaMiPe15, KaMiPe15bis}.  Generalizations of the periodicity conditions can be also exploited, for example, for the analysis of physical phenomena in layered media (see, for example, Brekhovskikh and Godin \cite{BrGo90} and Geis, S\"andig, and Mishuris \cite{GeSaMi03}).

It is therefore natural to seek a periodic counterpart to classical potential theory. The first step in this direction is constructing a periodic analog of the fundamental solutions. Once this is achieved, periodic layer potentials can be developed using integral kernels derived from these periodic fundamental solutions. The next task is to study their regularity properties and establish the corresponding jump formulas. With these tools in place, periodic boundary value problems can then be effectively addressed using potential-theoretic methods.

The objective of this paper is twofold. The first aim of this note is to collect and review some known constructions of periodic fundamental solutions for various partial differential operators. Subsequently, we will describe the main properties of the corresponding layer potentials. In doing so, we aim to provide a useful reference for future applications of potential theory to periodic problems. Specifically, we will consider the Laplace equation, the Helmholtz equation, the Lam\'e system, and the heat equation.

The second aim is to illustrate, with some new results, how an approach based on potential theory can be applied to the study of periodic domain perturbation problems. The method we use is the so-called {\it Functional Analytic Approach}, which was proposed by Lanza de Cristoforis both for singularly perturbed (see, e.g.,~\cite{La02}) and regularly perturbed (see, e.g.,~\cite{La07}) boundary value problems. For a detailed presentation of the method, we refer to the monograph \cite{DaLaMu21}.

Indeed, in both pure and applied mathematics, it is often important to understand how a quantity associated with the solutions of a boundary value problem depends on perturbations of the domain. Such perturbations can generally be classified into two types: regular perturbations, which preserve the regularity of the domain, and singular perturbations, which do not.

In this paper, regular perturbations are obtained as deformations of a reference shape through a diffeomorphism $\phi$. Different diffeomorphisms generally produce different shapes, and our objective is to understand how the function under consideration (e.g., the solution to a boundary value problem in the perturbed domain) depends on the choice of the diffeomorphism. This is a typical point of view in  the framework of {\it Shape Optimization} problems, where the goal is to find optimal configurations for certain shape functionals. Several examples can be found in the monographs by  {Henrot and Pierre \cite{HePi18}  and Soko\l owski and Zol\'esio \cite{SoZo92}}.   To study such perturbations, we adopt the method proposed by Lanza de Cristoforis (see, for example, \cite{La07}), with the aim of determining the regularity of the map that associates the diffeomorphism with the solution of the perturbed boundary value problem.
  
The primary advantage of Lanza's method, compared to other approaches found in the literature, is its ability to yield high-regularity results, such as real analyticity or {$C^\infty$}-smoothness. However, another approach has recently emerged, producing results on regularly perturbed problems that are comparable to those obtained using the method of Lanza de Cristoforis. For this, we refer the reader to the works on {\em Shape Holomorphy} by Henr\'iquez and Schwab \cite{HeSc21} and Jerez-Hankes, Schwab, and Zech \cite{JeScZe17}.

As mentioned, we will also study singular domain perturbations. Specifically, we consider boundary value problems in domains with small perforations of size \(\epsilon\) and analyze the behavior of the solution as \(\epsilon\) tends to \(0\). The most common approach for such problems is that of {\em Asymptotic Analysis}, which seeks to compute an asymptotic expansion of the solution for \(\epsilon\) close to \(0\) (see, for example, Maz'ya, Movchan, and Nieves \cite{MaMoNi13}, and Maz'ya, Nazarov, and Plamenevskij \cite{MaNaPl00a, MaNaPl00b}). In this work, however, we adopt the {\em Functional Analytic Approach} (see \cite{DaLaMu21}), aiming to represent the solution using real analytic functions, which can then be expanded into converging power series of $\epsilon$. For more details on the application of the {\it Functional Analytic Approach} to singularly perturbed periodic problems, we refer to \cite[Chapters 12 and 13]{DaLaMu21}. Concrete applications of such results include the study of dilute composite materials (see Movchan,  Movchan, and Poulton  \cite{MoMoPo02}) and {\it Topological Optimization} (see  Novotny and Soko\l owski \cite{NoSo13} and Novotny,  Soko\l owski, and  {Z}ochowski \cite{NoSoZo19}).

As mentioned above, we begin by presenting periodic fundamental solutions and periodic layer potentials for the Laplace equation, the Helmholtz equation, the Lam\'e system, and the heat equation. To proceed, we first introduce the geometric setting.

Let $n \in \mathbb{N}\setminus \{0,1\}$ be a natural number that represents the dimension of the space. To define the periodicity parameters we take 
\[
(q_{11},\ldots,q_{nn}) \in \mathopen]0,+\infty[^n
\] 
and define a (fundamental) cell $Q \subseteq \mathbb{R}^n$ and a matrix $q \in {\mathbb{D}}_{n}^{+}({\mathbb{R}})$ by
\[
Q := \prod_{j=1}^n \mathopen]0,q_{jj}[, \quad
q := 
 \begin{pmatrix}
  q_{11} & 0 & \cdots & 0 \\
  0 & q_{22} & \cdots & 0 \\
  \vdots  & \vdots  & \ddots & \vdots  \\
  0 & 0 & \cdots & q_{nn} 
 \end{pmatrix}\, .
\]
Here above {${\mathbb{D}}^+_{n}({\mathbb{R}})$} is the set of $n\times n$ diagonal matrices with {positive} real entries {on the diagonal}. 
We denote by $|Q|_n$  the $n$-dimensional measure of the cell $Q$, by $\nu_Q$ the outward unit normal to $\partial Q$, where it exists, and by $q^{-1}$ the inverse matrix of $q$.  

In order to define our periodic domains, we take a regularity parameter $\alpha\in\mathopen]0,1[$ and a subset  $\Omega_Q$  of $\mathbb{R}^n$ satisfying the following  assumption:
	\begin{equation}\label{OmegaQ_def}
	\Omega_Q\,\,\mbox{is a bounded open  subset of}\,\,\mathbb{R}^n\,\,\mbox{of class}\,C^{1,\alpha}\  \mbox{such that}\ \overline{\Omega_Q}\subseteq Q\, .
	\end{equation}

{For the definition and properties of functions and sets of the Schauder class $C^{m,\alpha}$, $m \in \mathbb{N}$, we refer to Gilbarg and Trudinger \cite[\S 4.1, \S 6.2]{GiTr01}}.
Then we define the following two periodic open sets:
\[
\mathbb{S}_q[\Omega_Q] := \bigcup_{z \in \mathbb{Z}^n}(qz+\Omega_Q ), 
\qquad \mathbb{S}_q[\Omega_Q]^- := \mathbb{R}^n \setminus \overline{\mathbb{S}_q[\Omega_Q]}.
\]

{As already said,} the aim of the paper is twofold: on the one hand, we are going to present the construction of periodic single and double-layer potentials in the domains $\mathbb{S}_q[\Omega_Q]$ and $\mathbb{S}_q[\Omega_Q]^-$ for different differential operators; on the other hand, we are going to show applications to the study of boundary value problems in perturbed periodic domains.

The paper is organized as follows. Section \ref{s:lap} contains the construction of the periodic analog of the fundamental solution of the Laplace equation and the properties of the corresponding periodic layer potentials. In Section \ref{s:helm}, we turn to consider the case of the Helmholtz equation: we first present quasi-periodic layer potentials and then in Subsection \ref{ss:singhelm} we show an application of quasi-periodic layer potentials to the study of a quasi-periodic Dirichlet problem for the equation $\Delta u+k^2u=0$ in a domain with a periodic set of small holes. Section \ref{s:lam} contains an example of periodic potential theory for the case of elliptic systems and more precisely we consider the Lam\'e equations. Finally, in Section \ref{s:heat}, we consider a parabolic equation: we present periodic layer potentials for the heat equation and then in Subsection \ref{ss:regheat} we apply our technique to a Dirichlet problem for the heat equation in a regularly perturbed domain.


\section{Periodic layer potentials for the Laplace equation}\label{s:lap}

{We begin by introducing the} periodic potential theory for the Laplace equation: we introduce a periodic analog of the fundamental solution and then we define the corresponding periodic layer potentials.

To do so, we need to define our notion of periodicity. Since the Helmholtz equation will require a notion of quasi-periodicity, we provide definitions for both periodicity and quasi-periodicity simultaneously. If $\mathcal{U}\subseteq \mathbb{R}^n$ is a subset such that $qz+\mathcal{U}\subseteq \mathcal{U}$ for all $z \in \mathbb{Z}^n$ and  if $\eta \in \mathbb{R}^n$, then a function $f$ from $\mathcal{U}$ to  $\mathbb{C}$ is said to be $\eta$-quasi-periodic with respect to $Q$, or simply $(Q,\eta)$-quasi-periodic, if  
\[
f(x+qe_h)e^{-i\eta \cdot (x+qe_h)} = f(x) e^{-i\eta \cdot x} \qquad \forall x\in {\mathcal{U}}, \,\forall h \in \{1,\ldots,n\},  
\]
where $e_1,\ldots,e_n$ is the canonical basis of $\mathbb{R}^n$.  If $\eta=0$, we {will  say that $f$ is $Q$-periodic (or simply periodic)}. {Note that, when $\eta \neq 0$, saying that a function $f$ is $(Q,\eta)$-quasi-periodic is equivalent to saying that $x \mapsto f(x)e^{-i\eta \cdot x}$ is $Q$-periodic.}

Before defining a periodic analog  of the fundamental solution, we need to introduce  {the notion of periodicity for distributions}.

For a function $f$  defined in ${\mathbb{R}}^{n}$ and $y\in{\mathbb{R}}^{n}$, we set
$\tau_{y}f(x):= f(x-y)$ for all $x\in {\mathbb{R}}^{n}$. If $u$ is a distribution in ${\mathbb{R}}^{n}$, then we set
\[
<\tau_{y}u,f>=<u,\tau_{-y}f>\qquad\forall f\in {\mathcal{D}}( {\mathbb{R}}^{n})\,,
\]
where ${\mathcal{D}}(\mathbb{R}^n)$ denotes the space of functions of  $C^{\infty}(\mathbb{R}^n)$
with compact support. We say that a distribution $G$ is {$(Q,\eta)$-quasi-periodic if
\[
\tau_{q_{jj}e_{j}}(G E_{-i\eta})=G E_{-i\eta}\qquad\forall j\in\{1,\dots,n\}\,,
\]
where, if $\xi \in \mathbb{C}^n$,  $E_{\xi}$ denotes the function
\[
E_{\xi }(x):= e^{\xi \cdot x}
\qquad  \forall x\in{\mathbb{R}}^{n}\, 
\]
(and we keep the same symbol for the associated distribution). As  before, if the distribution $G$ is $(Q,0)$-quasi-periodic, we simply say that $G$ is $Q$-periodic or periodic.}

In the following theorem, we present the construction of a periodic analog of a fundamental solution for the Laplace equation. For a proof we refer to \cite[Thms.~12.2 and 12.3]{DaLaMu21} (see also  Ammari and Kang~\cite[p.~53]{AmKa07}).

\begin{theorem}
\label{psper}
The generalized series
\[
S_{q,n}
:=
\sum_{z\in {\mathbb{Z}}^{n}\setminus \{0\}} \frac{1}{-4\pi^2 |q^{-1}z|^2|Q|_n}E_{ 2\pi iq^{-1}z }
\]
defines a tempered distribution in ${\mathbb{R}}^{n}$ such that $S_{q,n}$ is {$Q$-periodic} and such that 
\[
\Delta S_{q,n}=\sum_{z\in {\mathbb{Z}}^{n}}
\delta_{qz}-\frac{1}{|Q|_n}\,,
\]
where $\delta_{qz}$ denotes the Dirac measure with mass at $qz$, for all $z\in {\mathbb{Z}}^{n}$. Moreover, $S_{q,n}$ is real analytic in ${\mathbb{R}}^{n}\setminus q{\mathbb{Z}}^{n}$, $S_{q,n}\in L^{1}_{ {\mathrm{loc}} }({\mathbb{R}}^{n})$, and $S_{q,n}(x)=S_{q,n}(-x)$ for all $x \in \mathbb{R}^n \setminus q\mathbb{Z}^n$. 
\end{theorem}

\begin{remark}
We observe that a construction of a periodic analog of the fundamental solution for a general elliptic differential operator with constant coefficients can be carried out as it is done in Theorem \ref{psper} for $\Delta$. For a detailed proof, we refer again to \cite[Thm.~12.22]{DaLaMu21}, which indeed deals with the general case.
\end{remark}

By means of $S_{q,n}$ one can construct periodic layer potentials for the Laplace equation by simply replacing in the definition of layer potentials the (classical) fundamental solution for the Laplace equation with its periodic analog $S_{q,n}$.

Our framework will be the one of the Schauder classes and we accordingly introduce some notation. If $k\in \{0,1\}$ and $\beta\in\mathopen]0,1]$, then we introduce the spaces of real-valued periodic functions by setting:
 \index{$C^{k}_{q}(\overline{\mathbb{S}_q[\Omega_Q]} )$} \index{$C^{k,\beta}_{q}(\overline{\mathbb{S}_q[\Omega_Q]} )$} \index{$C^{k}_{q}(\overline{\mathbb{S}_q[\Omega_Q]^{-}})$}  \index{$C^{k,\beta}_{q}(\overline{\mathbb{S}_q[\Omega_Q]^{-}} )$}
\[
\begin{aligned}
&C^{k}_{q}(\overline{\mathbb{S}_q[\Omega_Q]} )
:=\left\{
u\in C^{k}(\overline{\mathbb{S}_q[\Omega_Q]} ):\,
u\ {\text{is $Q$-periodic}}
\right\}\,,\\
&C^{k,\beta}_{q}(\overline{\mathbb{S}_q[\Omega_Q]} )
:=\left\{
u\in C^{k,\beta}(\overline{\mathbb{S}_q[\Omega_Q]} ):\,
u\ {\text{is $Q$-periodic}}
\right\}\,,
\\
&C^{k}_{q}(\overline{\mathbb{S}_q[\Omega_Q]^{-}})
:=\left\{
u\in C^{k}(\overline{\mathbb{S}_q[\Omega_Q]^{-}}):\,
u\ {\text{is $Q$-periodic}}
\right\}\,,
\\
&C^{k,\beta}_{q}(\overline{\mathbb{S}_q[\Omega_Q]^{-}} )
:=\left\{
u\in C^{k,\beta}(\overline{\mathbb{S}_q[\Omega_Q]^{-}} ):\,
u\ {\text{is $Q$-periodic}}
\right\}\,.
\end{aligned}
\]
We consider $C^{k}_{q}(\overline{\mathbb{S}_q[\Omega_Q]} )$, $C^{k,\beta}_{q}(\overline{\mathbb{S}_q[\Omega_Q]} )$, $C^{k}_{q}(\overline{\mathbb{S}_q[\Omega_Q]^{-}})$, and $C^{k}_{q}(\overline{\mathbb{S}_q[\Omega_Q]^{-}})$ as Banach subspaces of $C^{k}_{b}(\overline{\mathbb{S}_q[\Omega_Q]} )$,  $C^{k,\beta}_{b}(\overline{\mathbb{S}_q[\Omega_Q]} )$,  $C^{k}_{b}(\overline{\mathbb{S}_q[\Omega_Q]^{-}})$, and  $C^{k,\beta}_{b}(\overline{\mathbb{S}_q[\Omega_Q]^{-}})$, respectively, where the subscript $b$ stands for `bounded.' For a precise definition we refer to \cite[Sections 2.6, 2.11, and 12.2]{DaLaMu21}. Clearly, when necessary, one can also provide the corresponding definition for complex valued functions.  We also set
\[
C^{k}_{q}(\overline{\mathbb{S}_q[\Omega_Q]},\mathbb{R}^n):=(C^{k}_{q}(\overline{\mathbb{S}_q[\Omega_Q]}))^n,\qquad  C^{k,\beta}_{q}(\overline{\mathbb{S}_q[\Omega_Q]} ,\mathbb{R}^n):=(C^{k,\beta}_{q}(\overline{\mathbb{S}_q[\Omega_Q]} ))^n,
\]
and 
\[
C^{k}_{q}(\overline{\mathbb{S}_q[\Omega_Q]^-},\mathbb{R}^n):=(C^{k}_{q}(\overline{\mathbb{S}_q[\Omega_Q]^-}))^n,\qquad  C^{k,\beta}_{q}(\overline{\mathbb{S}_q[\Omega_Q]^-} ,\mathbb{R}^n):=(C^{k,\beta}_{q}(\overline{\mathbb{S}_q[\Omega_Q]^-} )^n,
\]
for the corresponding spaces of vector-valued functions. Similarly, one defines $C^{k}_{q}(\mathbb{R}^n )$ and $C^{k,\beta}_{q}(\mathbb{R}^n)$.

We begin with the double-layer potential. Let $\alpha \in \mathopen ]0,1[$ and $\Omega_Q$ be as in \eqref{OmegaQ_def}. If  $\mu\in C^{0}(\partial\Omega_Q)$, then $\mathcal{D}_{q}[\partial\Omega_Q,\mu]$ denotes the {$Q$-periodic} double-layer potential with density $\mu$ defined by
\[
\mathcal{D}_{q}[\partial\Omega_Q,\mu](x):=-\int_{\partial\Omega_Q}\mu(y)\;\nu_{\Omega_Q}(y)\cdot\nabla S_{q,n}(x-y)\,d\sigma_y\qquad\forall x\in\mathbb{R}^n\,.
\]

 Moreover, we set
\[
\mathcal{K}_{q}[\partial\Omega_Q,\mu] := \mathcal{D}_{q}[\partial\Omega_Q,\mu]_{|\partial\Omega_Q} \qquad \mbox{ on } \partial\Omega_Q.
\]

In the following theorem, we collect some properties of the periodic double-layer potential for the Laplace equation. For a proof, we refer to \cite[Thm.~12.10]{DaLaMu21}.
\begin{theorem}
\label{perbvp.dbperpot}
Let $\alpha \in \mathopen ]0,1[$ and $\Omega_Q$ be as in \eqref{OmegaQ_def}. The following statements hold.
\begin{itemize}
\item[(i)] If $\mu\in C^{0}(\partial\Omega_Q)$, then the function $\mathcal{D}_{q}[\partial\Omega_Q,\mu]$ 
 is {$Q$-periodic} and 
\[
\Delta \mathcal{D}_{q}[\partial\Omega_Q,\mu](x)=0
\]
for all  $x\in 
 {\mathbb{R}}^{n}\setminus\partial{\mathbb{S}}_q[\Omega_Q]$.
\item[(ii)] If  $\mu \in C^{1,\alpha}(\partial \Omega_Q)$, then  the restriction $\mathcal{D}_{q}[\partial\Omega_Q,\mu]_{|\mathbb{S}_q[\Omega_Q]}$ can be extended to a continuous function $\mathcal{D}^{+}_{q}[\partial\Omega_Q,\mu] \in C^{1,\alpha}_q(\overline{\mathbb{S}_q[\Omega_Q]})$  and
\[
\mathcal{D}^{+}_{q}[\partial\Omega_Q,\mu] = \frac{1}{2}\mu+\mathcal{K}_{q}[\partial\Omega_Q,\mu]  \qquad \mbox{ on } \partial\Omega_Q\, .
\]
Moreover, the map from $C^{1,\alpha}(\partial \Omega_Q)$ to $ C^{1,\alpha}_q(\overline{\mathbb{S}_q[\Omega_Q]})$ 
that takes $\mu$ to $\mathcal{D}^{+}_{q}[\partial\Omega_Q,\mu]$ is linear and continuous.
\item[(iii)] If  $\mu \in C^{1,\alpha}(\partial \Omega_Q)$, then  the restriction $\mathcal{D}_{q}[\partial\Omega_Q,\mu]_{|\mathbb{S}_q[\Omega_Q]^-}$ can be extended to a continuous function $\mathcal{D}^{-}_{q}[\partial\Omega_Q,\mu] \in C^{1,\alpha}_q(\overline{\mathbb{S}_q[\Omega_Q]^-})$ and
\[
\mathcal{D}^{-}_{q}[\partial\Omega_Q,\mu] = -\frac{1}{2}\mu+\mathcal{K}_{q}[\partial\Omega_Q,\mu]  \qquad \mbox{ on } \partial\Omega_Q\, .
\]
Moreover, the map from $C^{1,\alpha}(\partial \Omega_Q)$ to $ C^{1,\alpha}_q(\overline{\mathbb{S}_q[\Omega_Q]^-})$ 
that takes $\mu$ to $\mathcal{D}^{-}_{q}[\partial\Omega_Q,\mu]$ is linear and continuous.
\item[(iv)] The map that takes $\mu \in C^{1,\alpha}(\partial\Omega_Q)$ to $\mathcal{K}_{q}[\partial\Omega_Q,\mu]$ is a compact operator from  $C^{1,\alpha}(\partial\Omega_Q)$ 
to itself.
\end{itemize}
\end{theorem}

As we have done for the periodic double-layer potential, we can define the periodic single-layer potential for the Laplace equation. Let $\alpha \in \mathopen ]0,1[$ and $\Omega_Q$ be as in \eqref{OmegaQ_def}. If $\mu\in C^{0}(\partial\Omega_Q)$,  then  $\mathcal{S}_{q}[\partial\Omega_Q,\mu]$ denotes the {$Q$-periodic} single-layer potential with density $\mu$ defined by
\[
\mathcal{S}_{q}[\partial\Omega_Q,\mu](x):=\int_{\partial\Omega_Q}S_{q,n}(x-y)\mu(y)\,d\sigma_y\qquad\forall x\in\mathbb{R}^n\, .
\] 

Moreover, we set 
\[
\left(\mathcal{K}_{q}\right)^*[\partial\Omega_Q,\mu](x) := \int_{\partial \Omega_Q}\nu_{\Omega_Q}(x) \cdot\nabla S_{q,n}(x-y)\mu(y)\,d\sigma_y \qquad \forall x \in \partial\Omega_Q.
\]

The following theorem presents some basic properties of the periodic single-layer potential (see \cite[Thm.~12.8]{DaLaMu21} for a proof).

\begin{theorem}
\label{perbvp.sperpot}
Let $\alpha \in \mathopen]0,1[$ and $\Omega_Q$ be as in \eqref{OmegaQ_def}. The following statements hold.
\begin{itemize}
\item[(i)] If $\mu\in C^{0}(\partial\Omega_Q)$, then the function $\mathcal{S}_{q}[\partial\Omega_Q,\mu]$ 
is continuous. Moreover, $\mathcal{S}_{q}[\partial\Omega_Q,\mu]$ is {$Q$-periodic} and 
\[
\Delta \mathcal{S}_{q}[\partial\Omega_Q,\mu](x)=
-\frac{1}{|Q|_n}\int_{\partial \Omega_Q}\mu\, d\sigma
\]
for all  $x\in 
 {\mathbb{R}}^{n}\setminus\partial{\mathbb{S}}_q[\Omega_Q]$.
 \item[(ii)] If  $\mu \in C^{0,\alpha}(\partial \Omega_Q)$, then the function $\mathcal{S}^{+}_{q}[\partial\Omega_Q,\mu] := \mathcal{S}_{q}[\partial\Omega_Q,\mu]_{|\overline{\mathbb{S}_q[\Omega_Q]}}$ belongs to $C^{1,\alpha}_q(\overline{\mathbb{S}_q[\Omega_Q]})$ and
\[
\nu_{\Omega_Q}(x) \cdot \nabla\mathcal{S}^{+}_{q}[\partial\Omega_Q,\mu](x) = -\frac{1}{2}\mu(x)+\left(\mathcal{K}_{q}\right)^*[\partial\Omega_Q,\mu](x) \qquad \forall x \in \partial\Omega_Q.
\]
Moreover, the map from $C^{0,\alpha}(\partial \Omega_Q)$ to $ C^{1,\alpha}_q(\overline{\mathbb{S}_q[\Omega_Q]})$ 
that takes $\mu$ to $\mathcal{S}^{+}_{q}[\partial\Omega_Q,\mu]$ is linear and continuous.
\item[(iii)] If  $\mu \in C^{0,\alpha}(\partial \Omega_Q)$, then the function $\mathcal{S}^{-}_{q}[\partial\Omega_Q,\mu] := \mathcal{S}_{q}[\partial\Omega_Q,\mu]_{|\overline{\mathbb{S}_q[\Omega_Q]^-}}$ belongs to $C^{1,\alpha}_q(\overline{\mathbb{S}_q[\Omega_Q]^-})$ and
\[
\nu_{\Omega_Q}(x) \cdot \nabla\mathcal{S}^{-}_{q}[\partial\Omega_Q,\mu](x) = \frac{1}{2}\mu(x)+\left(\mathcal{K}_{q}\right)^*[\partial\Omega_Q,\mu](x) \qquad \forall x \in \partial\Omega_Q.
\]
Moreover, the map from $C^{0,\alpha}(\partial \Omega_Q)$ to $ C^{1,\alpha}(\overline{\mathbb{S}_q[\Omega_Q]^-})$ 
that takes $\mu$ to $\mathcal{S}^{-}_{q}[\partial\Omega_Q,\mu]$ is linear and continuous.
\item[(iv)] The map that takes $\mu \in C^{0,\alpha}(\partial\Omega_Q)$ to $\left(\mathcal{K}_{q}\right)^*[\partial\Omega_Q,\mu]$ is a compact operator from  $C^{0,\alpha}(\partial\Omega_Q)$ to itself.
 \end{itemize}
\end{theorem}

{\begin{remark}
Theorem \ref{perbvp.sperpot} (i) implies in particular that if we want $\mathcal{S}_{q}[\partial\Omega_Q,\mu]$ to be harmonic, we need to take a density $\mu$ such that $\int_{\partial \Omega_Q}\mu\, d\sigma=0$.
\end{remark}}

We note that periodic layer potentials for the Laplace equation have been used, for example, in \cite{DaMu13, LaMu13, Mu12} for the analysis of linear and nonlinear boundary value problems for the Laplace equation in an unbounded domain with small periodic perforations.

\section{Quasi-periodic layer potentials for the Helmholtz equation}\label{s:helm}

{We pass to} present quasi-periodic layer potentials for the Helmholtz equation {and their properties}.  {After that,  we also} show an application to the study of a quasi-periodic boundary value problem for the equation  $\Delta u +k^2 u =0$  in a singularly perturbed periodic domain.

For $k \in \mathbb{C}$ and $\eta \in \mathbb{R}^n$ we are interested in solving boundary value problems for  the   Helmholtz equation $\Delta u +k^2 u =0$ for $(Q,\eta)$-quasi-periodic functions. In this section, we will consider complex-valued functions.
To define $(Q,\eta)$-quasi-periodic layer potentials for such equation, we now construct a $(Q,\eta)$-quasi-periodic distribution that will play the role of the fundamental solution of the Helmholtz operator $\Delta+k^2$. A proof can be found in \cite[Props.~2.4 and 2.5]{BrDaLuMu24}. For this construction, see also  Ammari, Kang and Lee \cite[p. 123]{AmKaLe09}, Linton \cite{Li98},  Poulton, Botten, McPhedran and Movchan \cite{PoBoMcMo99}.

\begin{theorem}\label{qnSR}
Let  $\eta \in \mathbb{R}^n$ and $k \in \mathbb{C}$. Let  
\[
Z_{q,\eta}(k) := \left\{z \in \mathbb{Z}^n : k^2 = |2\pi q^{-1}z +\eta|^2\right\}.
\]
Then the set $Z_{q,\eta}(k)$ is finite and the generalized series
\begin{equation*}
G^{k}_{q,\eta} := \sum_{z \in \mathbb{Z}^n \setminus Z_{q,\eta}(k)  } \frac{1}{|Q|_n(k^2-|2\pi q^{-1}z +\eta|^2)}{E_{i(2\pi q^{-1}z +\eta)}}
\end{equation*}
defines a tempered distribution. Moreover, $G^{k}_{q,\eta}$  is $(Q,\eta)$-quasi-periodic  in the sense of distributions and 
\[
(\Delta+k^2)G^{k}_{q,\eta} = \sum_{z \in \mathbb{Z}^n}\delta_{qz}e^{iqz\cdot \eta} - \sum_{z \in Z_{q,\eta}(k)}\frac{1}{|Q|_n}{E_{i(2\pi q^{-1}z +\eta)}}\, ,
\]
$G^{k}_{q,\eta}$  is real analytic    in $\mathbb{R}^n\setminus q\mathbb{Z}^n$, and $G^{k}_{q,\eta} \in L^1_{\mathrm{loc}}(\mathbb{R}^n)$.
\end{theorem}

As we have done for the Laplace equation, we define quasi-periodic layer potentials for the Helmholtz equation. Let $\alpha \in \mathopen]0,1[$ and $\Omega_Q$ be as in \eqref{OmegaQ_def}. Let  $\eta \in \mathbb{R}^n$ and $k \in \mathbb{C}$.    
 We now introduce   $(Q,\eta)$-quasi-periodic layer potentials for the Helmholtz equation.
We start with the double layer potential: if $\mu \in C^0(\partial\Omega_Q)$, then the $(Q,\eta)$-quasi-periodic double-layer potential for the Helmholtz equation is 
\[
\mathcal{D}^{k}_{q,\eta}[\partial\Omega_Q,\mu](x) := -\int_{\partial \Omega_Q}\nu_{\Omega_Q}(y) \cdot\nabla G^{k}_{q,\eta}(x-y)\mu(y)\,d\sigma_y \qquad \forall x \in \mathbb{R}^n.
\]
 Moreover, we set
\[
\mathcal{K}^{k}_{q,\eta}[\partial\Omega_Q,\mu] := \mathcal{D}^{k}_{q,\eta}[\partial\Omega_Q,\mu]_{|\partial\Omega_Q} \qquad \mbox{ on } \partial\Omega_Q.
\]

We collect some properties of $\mathcal{D}^{k}_{q,\eta}[\partial\Omega_Q,\mu]$ in the following theorem (for a proof see  \cite[Prop.~3.1]{BrDaLuMu24}).

\begin{theorem}\label{prop:dlphelm}
Let $\alpha \in \mathopen]0,1[$ and $\Omega_Q$ be as in \eqref{OmegaQ_def}. Let $\eta \in \mathbb{R}^n$ and $k \in \mathbb{C}$. Then the following statements hold.
\begin{itemize}
\item[(i)] If  $\mu \in C^{0}(\partial \Omega_Q)$, then the function $\mathcal{D}^{k}_{q,\eta}[\partial\Omega_Q,\mu]$ is $(Q,\eta)$-quasi-periodic,  $\mathcal{D}^{k}_{q,\eta}[\partial\Omega_Q,\mu]$ is of class $C^\infty(\mathbb{R}^n\setminus \partial \mathbb{S}_q[\Omega_Q])$ and
\begin{align*}
(\Delta&+k^2)\mathcal{D}^{k}_{q,\eta}[\partial\Omega_Q,\mu](x) \\
&= \frac{1}{|Q|_n}\sum_{z \in Z_{q,\eta}(k)}e^{ix\cdot (2\pi  q^{-1}z+\eta)}i(2\pi  q^{-1}z+\eta)\cdot \int_{\partial \Omega_Q}\nu_{\Omega_Q}(y)e^{-iy\cdot (2\pi  q^{-1}z+\eta)} \mu(y)\,d\sigma_y
\end{align*}
for all $x \in \mathbb{R}^n \setminus \partial \mathbb{S}_q[\Omega_Q]$.
\item[(ii)] If  $\mu \in C^{1,\alpha}(\partial \Omega_Q)$, then  the restriction $\mathcal{D}^{k}_{q,\eta}[\partial\Omega_Q,\mu]_{|\mathbb{S}_q[\Omega_Q]}$ can be extended to a continuous function $\mathcal{D}^{k,+}_{q,\eta}[\partial\Omega_Q,\mu] \in C^{1,\alpha}(\overline{\mathbb{S}_q[\Omega_Q]})$  and
\[
\mathcal{D}^{k,+}_{q,\eta}[\partial\Omega_Q,\mu] = \frac{1}{2}\mu+\mathcal{K}^{k}_{q,\eta}[\partial\Omega_Q,\mu]  \qquad \mbox{ on } \partial\Omega_Q\, .
\]
Moreover, the map from $C^{1,\alpha}(\partial \Omega_Q)$ to $ C^{1,\alpha}(\overline{\mathbb{S}_q[\Omega_Q]})$ 
that takes $\mu$ to $\mathcal{D}^{k,+}_{q,\eta}[\partial\Omega_Q,\mu]$ is linear and continuous.
\item[(iii)] If  $\mu \in C^{1,\alpha}(\partial \Omega_Q)$, then  the restriction $\mathcal{D}^{k}_{q,\eta}[\partial\Omega_Q,\mu]_{|\mathbb{S}_q[\Omega_Q]^-}$ can be extended to a continuous function $\mathcal{D}^{k,-}_{q,\eta}[\partial\Omega_Q,\mu] \in C^{1,\alpha}(\overline{\mathbb{S}_q[\Omega_Q]^-})$ and
\[
\mathcal{D}^{k,-}_{q,\eta}[\partial\Omega_Q,\mu] = -\frac{1}{2}\mu+\mathcal{K}^{k}_{q,\eta}[\partial\Omega_Q,\mu]  \qquad \mbox{ on } \partial\Omega_Q\, .
\]
Moreover, the map from $C^{1,\alpha}(\partial \Omega_Q)$ to $ C^{1,\alpha}(\overline{\mathbb{S}_q[\Omega_Q]^-})$ 
that takes $\mu$ to $\mathcal{D}^{k,-}_{q,\eta}[\partial\Omega_Q,\mu]$ is linear and continuous.
\item[(iv)] The map that takes $\mu \in C^{1,\alpha}(\partial\Omega_Q)$ to $\mathcal{K}^{k}_{q,\eta}[\partial\Omega_Q,\mu]$ is a compact operator from  $C^{1,\alpha}(\partial\Omega_Q)$ 
to itself.
\end{itemize}
\end{theorem}

Let $\alpha \in \mathopen]0,1[$ and $\Omega_Q$ be as in \eqref{OmegaQ_def}. Let  $\eta \in \mathbb{R}^n$ and $k \in \mathbb{C}$.   Then, for $\mu \in C^0(\partial\Omega_Q)$, we introduce the $(Q,\eta)$-quasi-periodic single-layer potential for the Helmholtz equation: 
\[
\mathcal{S}^{k}_{q,\eta}[\partial\Omega_Q,\mu](x) := \int_{\partial \Omega_Q}  G^{k}_{q,\eta}(x-y)\mu(y)\,d\sigma_y \qquad \forall x \in \mathbb{R}^n.
\]
Moreover, we set 
\[
\left(\mathcal{K}^{k}_{q,\eta}\right)^*[\partial\Omega_Q,\mu](x) := \int_{\partial \Omega_Q}\nu_{\Omega_Q}(x) \cdot\nabla G^{k}_{q,\eta}(x-y)\mu(y)\,d\sigma_y \qquad \forall x \in \partial\Omega_Q.
\]

In the following theorem, we collect some properties of $\mathcal{S}^{k}_{q,\eta}[\partial\Omega_Q,\mu]$ (see \cite[Prop.~3.2]{BrDaLuMu24} for a proof).

\begin{theorem}\label{prop:slphelm}
Let $\alpha \in \mathopen]0,1[$ and $\Omega_Q$ be as in \eqref{OmegaQ_def}. Let $\eta \in \mathbb{R}^n$ and $k \in \mathbb{C}$. Let $\Omega_Q$ be a bounded open subset of $\mathbb{R}^n$ of class $C^{1,\alpha}$  such that $\overline{\Omega_Q} \subseteq Q$.   Then the following statements hold.
\begin{itemize}
\item[(i)]  If  $\mu \in {C^{0}(\partial \Omega_Q)}$, then the function $\mathcal{S}^{k}_{q,\eta}[\partial\Omega_Q,\mu]$  is  {$(Q,\eta)$-quasi-periodic}, is continuous in $\mathbb{R}^n$, is of class $C^\infty(\mathbb{R}^n\setminus \partial \mathbb{S}_q[\Omega_Q])$, and
\begin{align*}
(\Delta&+k^2)\mathcal{S}^{k}_{q,\eta}[\partial\Omega_Q,\mu](x) \\
&= -\frac{1}{|Q|_n}\sum_{z \in Z_{q,\eta}(k)}e^{ix\cdot (2\pi  q^{-1}z+\eta)} \int_{\partial \Omega_Q}e^{-iy\cdot (2\pi  q^{-1}z+\eta)} \mu(y)\,d\sigma_y\\
\end{align*}
for all $x \in \mathbb{R}^n\setminus \partial \mathbb{S}_q[\Omega_Q]$.
\item[(ii)] If  $\mu \in C^{0,\alpha}(\partial \Omega_Q)$, then the function $\mathcal{S}^{k,+}_{q,\eta}[\partial\Omega_Q,\mu] := \mathcal{S}^{k}_{q,\eta}[\partial\Omega_Q,\mu]_{|\overline{\mathbb{S}_q[\Omega_Q]}}$ belongs to $C^{1,\alpha}(\overline{\mathbb{S}_q[\Omega_Q]})$ and
\[
\nu_{\Omega_Q}(x) \cdot \nabla\mathcal{S}^{k,+}_{q,\eta}[\partial\Omega_Q,\mu](x) = -\frac{1}{2}\mu(x)+\left(\mathcal{K}^{k}_{q,\eta}\right)^*[\partial\Omega_Q,\mu](x) \qquad \forall x \in \partial\Omega_Q.
\]
Moreover, the map from $C^{0,\alpha}(\partial \Omega_Q)$ to $ C^{1,\alpha}(\overline{\mathbb{S}_q[\Omega_Q]})$ 
that takes $\mu$ to $\mathcal{S}^{k,+}_{q,\eta}[\partial\Omega_Q,\mu]$ is linear and continuous.
\item[(iii)] If  $\mu \in C^{0,\alpha}(\partial \Omega_Q)$, then the function $\mathcal{S}^{k,-}_{q,\eta}[\partial\Omega_Q,\mu] := \mathcal{S}^{k}_{q,\eta}[\partial\Omega_Q,\mu]_{|\overline{\mathbb{S}_q[\Omega_Q]^-}}$ belongs to $C^{1,\alpha}(\overline{\mathbb{S}_q[\Omega_Q]^-})$ and
\[
\nu_{\Omega_Q}(x) \cdot \nabla\mathcal{S}^{k,-}_{q,\eta}[\partial\Omega_Q,\mu](x) = \frac{1}{2}\mu(x)+\left(\mathcal{K}^{k}_{q,\eta}\right)^*[\partial\Omega_Q,\mu](x) \qquad \forall x \in \partial\Omega_Q.
\]
Moreover, the map from $C^{0,\alpha}(\partial \Omega_Q)$ to $ C^{1,\alpha}(\overline{\mathbb{S}_q[\Omega_Q]^-})$ 
that takes $\mu$ to $\mathcal{S}^{k,-}_{q,\eta}[\partial\Omega_Q,\mu]$ is linear and continuous.
\item[(iv)] The map that takes $\mu \in C^{0,\alpha}(\partial\Omega_Q)$ to $\left(\mathcal{K}^{k}_{q,\eta}\right)^*[\partial\Omega_Q,\mu]$ is a compact operator from  $C^{0,\alpha}(\partial\Omega_Q)$ to itself.
\end{itemize}
\end{theorem}

{\begin{remark}
Theorems \ref{prop:dlphelm} (i) and \ref{prop:slphelm} (i) imply that if $Z_{q,\eta}(k)=\emptyset$, then both $\mathcal{D}^{k}_{q,\eta}[\partial\Omega_Q,\mu]$ and $\mathcal{S}^{k}_{q,\eta}[\partial\Omega_Q,\mu]$ satisfy the Helmholtz equation in $\mathbb{R}^n\setminus \partial \mathbb{S}_q[\Omega_Q]$.

\end{remark}}

\subsection{A singularly perturbed boundary value problem} \label{ss:singhelm}

In this section, we show an application of quasi-periodic layer potentials  to the study of the asymptotic behavior of the solution to a singularly perturbed boundary value problem for the Helmholtz equation in the framework of quasi-periodic functions.

For simplicity, we assume in this Section \ref{ss:singhelm} that
\[
n=3\, .
\]
In the general case, where $n \in \mathbb{N}\setminus \{0,1\}$, the asymptotic behavior of the solution strongly depends on the dimension $n$. For an illustration of this, we refer to \cite[\S 6]{BrDaLuMu24}, where a nonlinear Robin problem for the Helmholtz equation is analyzed. 

However, taking $n=3$ simplifies the computations and allows for a more direct and illustrative application of our technique.

 Our problem will be set in a periodic domain obtained by removing from $\mathbb{R}^3$ a periodic set of small holes. To define such a domain, we take $\alpha\in\mathopen]0,1[$  and a subset  $\Omega$  of $\mathbb{R}^3$ satisfying the following  assumption:
	\begin{equation}\label{Omega_def}
	\begin{split}
	&\Omega\,\,\mbox{is a bounded open connected subset of}\,\,\mathbb{R}^3\,\,\mbox{of class}\,C^{1,\alpha} \\
	&\mbox{such that}\   \mathbb{R}^3 \setminus{\overline{\Omega}}\,\,\mbox{is connected.}
	\end{split}
	\end{equation}
Let $p\in Q$. Then there exists $\epsilon_0\in\mathopen]0,+\infty[$ such that
	\begin{equation}\label{epsilon_0}
	  \Omega_{p,\epsilon}:=p+\epsilon\,{\overline{\Omega}}\subseteq Q \quad\forall\epsilon\in\mathopen]-\epsilon_0, \epsilon_0[\, .
	\end{equation}	
We will remove from $\mathbb{R}^3$ a set of periodic copies of $\Omega_{p,\epsilon}$. 
{From now on,  $\eta \in \mathbb{R}^3$ will be fixed}. To define the wave number for the Helmholtz equation we take {$k \in \mathbb{C}$  such that 
\[
k^2 \notin \left\{|2\pi q^{-1}z+\eta|^2 : z \in \mathbb{Z}^3\right\}\, ,
\]
i.e., such that $k^2$ does not belong to  the spectrum of the  Laplacian $-\Delta$ acting on $(Q,\eta)$-quasi-periodic functions {(see \cite[\S 4]{BrDaLuMu24})}. Let $g \in C^{1,\alpha}(\partial\Omega)$. Then, for $\epsilon\in\mathopen]0,\epsilon_{0}[$, we consider the following Dirichlet problem
\begin{equation}\label{bvpe}
\begin{cases}
\Delta u + k^2 u=0 \qquad &\mbox{ in } \mathbb{S}_q[\Omega_{p,\epsilon}]^-,\\
u \, \mbox{is $(Q,\eta)$-quasi-periodic},\\
u(x)=g((x-p)/\epsilon) \qquad &\forall x \in  \partial\Omega_{p,\epsilon}.
\end{cases}
\end{equation}

For $\epsilon$ positive and small enough, problem \eqref{bvpe} has a unique solution and such a solution can be represented in terms of a single-layer potential. We demonstrate this fact in the following proposition.

\begin{proposition}\label{prop:bij}
Let $\alpha \in \mathopen]0,1[$ and $\Omega$ be as in assumption \eqref{Omega_def}. Let $p \in Q$. Let $\epsilon_0$ be as in assumption \eqref{epsilon_0}. Let $\eta \in \mathbb{R}^3$ and let {$k \in \mathbb{C}$ be such} that $k^2 \notin \left\{|2\pi q^{-1}z+\eta|^2 : z \in \mathbb{Z}^3\right\}$.     Let $g \in C^{1,\alpha}(\partial\Omega)$. Then there exists $\epsilon^\#_0 \in  \mathopen]0,\epsilon_0[$ such that for each  $\epsilon \in \mathopen]0,\epsilon^\#_0[$,
problem \eqref{bvpe} has a unique solution $u_\epsilon \in C^{1,\alpha}(\overline{\mathbb{S}_q[\Omega_{p,\epsilon}]^-})$ and such solution is given by the function
 \[
u_\epsilon(x):= {\mathcal{S}^{k,-}_{q,\eta}}[\partial\Omega_{p,\epsilon},\epsilon^{-1}\theta_\epsilon((\cdot-p)/\epsilon)](x) \qquad \forall x \in \overline{\mathbb{S}_q[\Omega_{p,\epsilon}]^-}\, ,
 \]
where 
$\theta_\epsilon \in 
C^{0,\alpha}(\partial\Omega)$ is the unique solution of the equation
\begin{equation}\label{eq:bij}
\begin{split}
\epsilon^{-1}\mathcal{S}^{k}_{q,\eta} &[\partial\Omega_{p,\epsilon},\theta((\cdot-p)/\epsilon)](p+\epsilon t)=g(t) \qquad \forall t \in \partial \Omega\, .
\end{split}
\end{equation}
\end{proposition}

\begin{proof}
We can verify that \( k^2 \) is not a \((Q, \eta)\)-quasi-periodic Dirichlet eigenvalue for \( -\Delta \) in \( \mathbb{S}_q[\Omega_{p,\epsilon}]^- \) for \( \epsilon \) sufficiently close to 0. This can be proven by leveraging the fact that \( k^2 \) does not belong to the spectrum of the Laplacian \( -\Delta \) acting on \((Q, \eta)\)-quasi-periodic functions in \( \mathbb{R}^n \), and by following the same argument as in \cite{BrDaLuMu24}, which is a straightforward modification of a proof by Rauch and Taylor in \cite{RaTa75}. Then, by \cite[Thm.~4.1]{BrDaLuMu24}, we know that there exists $\epsilon^\#_0 \in  \mathopen]0,\epsilon_0[$ such that problem \eqref{bvpe} has a unique solution $u_\epsilon \in C^{1,\alpha}(\overline{\mathbb{S}_q[\Omega_{p,\epsilon}]^-})$ for each  $\epsilon \in \mathopen]0,\epsilon^\#_0[$.

 Possibly shrinking $\epsilon^\#_0$, a similar argument, which is also  derived from Rauch and Taylor \cite{RaTa75} (see \cite{BrDaLuMu24}), proves that $k^2$ is not a $(Q,\eta)$-quasi-periodic Neumann eigenvalue for $-\Delta$ in $\mathbb{S}_q[\Omega_{p,\epsilon}]^-$ for $\epsilon \in \mathopen]-\epsilon^\#_0,\epsilon^\#_0[$. Also, for a possibly smaller $\epsilon^\#_0$, $k^2$ is not a Dirichlet eigenvalue of  $-\Delta$ in $\Omega_{p,\epsilon}$ for all $\epsilon \in ]0,\epsilon^\#_0[$. Then, for each  $\epsilon \in \mathopen]0,\epsilon^\#_0[$,  we can prove by  \cite[Cor.~4.5]{BrDaLuMu24} that there exists a unique $\theta_\epsilon \in 
C^{0,\alpha}(\partial\Omega)$ such that 
 \[
u_\epsilon(x)= {\mathcal{S}^{k,-}_{q,\eta}}[\partial\Omega_{p,\epsilon},\epsilon^{-1}\theta_\epsilon((\cdot-p)/\epsilon)](x) \qquad  \forall x \in \   \overline{\mathbb{S}_q[\Omega_{p,\epsilon}]^-}\, .
 \]
The function $\theta_\epsilon$ is the unique solution in $C^{0,\alpha}(\partial\Omega)$ of  equation \eqref{eq:bij}. 
\end{proof}

Thanks to Proposition \ref{prop:bij}, the problem of understanding the dependence of the solution to \eqref{bvpe} upon the parameter $\epsilon$ is reduced to that of understanding the dependence of the solution to the integral equation \eqref{eq:bij} upon  $\epsilon$.

We set 
\[
\mathcal{S}[\partial\Omega,\theta](t):= \int_{\partial \Omega}\frac{1}{-4\pi |t-s|}\theta(s)\,d\sigma_s \qquad \forall t \in \partial \Omega\, ,
\]
for all $\theta \in C^{0,\alpha}(\partial \Omega)$. {The operator $\mathcal{S}[\partial\Omega,\cdot]$ is the harmonic (non-periodic) single-layer potential operator and it is well} known  that $\mathcal{S}[\partial\Omega,\cdot]$ is a linear homeomorphism from $C^{0,\alpha}(\partial \Omega)$ to $C^{1,\alpha}(\partial \Omega)$ (see \cite[Thm.~6.46]{DaLaMu21}).  Then we observe that, possibly shrinking $\epsilon^\#_0$, there exists a real analytic map $M$ from $\mathopen]-\epsilon^\#_0,\epsilon^\#_0[$ to $\mathcal{L}(C^{0,\alpha}(\partial \Omega),C^{1,\alpha}(\partial \Omega))$ such that   equation \eqref{eq:bij} can be rewritten as
\begin{equation}\label{eq:bijbis}
\begin{split}
&\mathcal{S}[\partial\Omega,\theta]+ \epsilon M[\epsilon](\theta) -g=0 \qquad \mathrm{on}\ \partial \Omega
\end{split}
\end{equation} for all $\epsilon \in \mathopen]0,\epsilon^\#_0[$ (cf.~\cite[Prop.~5.1 and Cor.~5.2]{BrDaLuMu24}). 
We introduce the map $\Lambda$ from $$\mathopen]-\epsilon^\#_0,\epsilon^\#_0\mathclose[  \times C^{0,\alpha}(\partial \Omega)$$ to $C^{1,\alpha}(\partial \Omega)$ defined by
\[
\Lambda[\epsilon,\theta]:=\mathcal{S}[\partial\Omega,\theta]+ \epsilon M[\epsilon](\theta) -g \qquad \mathrm{on}\ \partial \Omega\, ,
\]
for all $(\epsilon,\theta)\in \mathopen]-\epsilon^\#_0,\epsilon^\#_0\mathclose[ \times C^{0,\alpha}(\partial \Omega)$ (cf.~\eqref{eq:bijbis}). Then, if $\epsilon \in \mathopen]0,\epsilon^\#_0[$,  we can rewrite equation \eqref{eq:bijbis}  as
\begin{equation}\label{eq:Lambdaeps}
\Lambda[\epsilon,\theta]=0\, .
\end{equation}
We now wish to understand the behavior of the solutions $\theta$ of equation \eqref{eq:Lambdaeps} when $\epsilon$ is close to the degenerate value $\epsilon=0$. We first note that, letting $\epsilon \to 0^+$ in equation \eqref{eq:Lambdaeps}, we obtain
\[
\Lambda[0,\theta]=0\, ,
\]
that is
\begin{equation}\label{eq:lim}
\begin{split}
&\mathcal{S}[\partial\Omega,\theta] = g \qquad \mathrm{on}\ \partial \Omega\, .
\end{split}
\end{equation}
By \cite[Thm.~6.46]{DaLaMu21}, equation \eqref{eq:lim} is well known to have a unique solution $\tilde{\theta}$  in $C^{0,\alpha}(\partial \Omega)$.

In the following proposition, we study the dependence of the solution $\theta$ of equation \eqref{eq:Lambdaeps} upon the parameter $\epsilon$ in a small neighborhood of the degenerate value $\epsilon=0$.

\begin{proposition}\label{prop:Theta}
Let $\alpha \in \mathopen]0,1[$ and $\Omega$ be as in assumption \eqref{Omega_def}. Let $p \in Q$. Let $\epsilon_0$ be as in assumption \eqref{epsilon_0}. Let $\eta \in \mathbb{R}^3$ and let {$k \in \mathbb{C}$ be such} that $k^2 \notin \left\{|2\pi q^{-1}z+\eta|^2 : z \in \mathbb{Z}^3\right\}$.      Let $g \in C^{1,\alpha}(\partial\Omega)$. Let $\epsilon^\#_0 $ be as in Proposition \ref{prop:bij}. Let $\tilde{\theta}$ be the unique solution in $C^{0,\alpha}(\partial \Omega)$ of equation \eqref{eq:lim}.  Let $\theta_\epsilon$ be as in Proposition \ref{prop:bij}. Then there exist $\epsilon' \in \mathopen]0,\epsilon^\#_0[$, an open neighborhood $\mathcal{O}_{\tilde{\theta}}$ of $\tilde{\theta}$ in $C^{0,\alpha}(\partial \Omega)$, and a real analytic  map $\Theta$ from $\mathopen]-\epsilon',\epsilon'\mathclose[$ to $\mathcal{O}_{\tilde{\theta}}$ such that  the set of zeros of $\Lambda$ in $\mathopen]-\epsilon',\epsilon'\mathclose[\times \mathcal{O}_{\tilde{\theta}}$ coincides with the graph of $\Theta$. In particular,  
\[
\Theta[0]=\tilde{\theta}\, , \qquad \Theta[\epsilon]=\theta_\epsilon \qquad \forall \epsilon \in \mathopen]0,\epsilon'\mathclose[\, .
\]
\end{proposition}
\begin{proof}
We wish to apply  the Implicit Function Theorem for real analytic maps in Banach spaces. Possibly shrinking $\epsilon^\#_0$ as in the discussion after the proof of Proposition \ref{prop:bij} and using \cite[Prop.~5.1 and Cor.~5.2]{BrDaLuMu24}, we can demonstrate that $\Lambda$ is a real analytic map from $\mathopen]-\epsilon^\#_0,\epsilon^\#_0\mathclose[ \times C^{0,\alpha}(\partial \Omega)$ to $C^{1,\alpha}(\partial \Omega)$. By standard calculus in Banach spaces,  the partial differential $d_\theta \Lambda [0,\tilde{\theta}]$ of $\Lambda$ at $(0,\tilde{\theta})$ with respect to $\theta$ is given by the linear map
\[
d_\theta \Lambda [0,\tilde{\theta}]=\mathcal{S}[\partial\Omega,\cdot]\, .
\]
By \cite[Thm.~6.46]{DaLaMu21}, $d_\theta \Lambda [0,\tilde{\theta}]$ is a linear homeomorphism from $C^{0,\alpha}(\partial \Omega)$ to $C^{1,\alpha}(\partial \Omega)$ and, therefore,  the Implicit Function Theorem for real analytic maps in Banach spaces (cf.~Deimling \cite[Thm.~15.3]{De85}) implies the existence of  $\epsilon' \in \mathopen]0,\epsilon^\#_0[$, an open neighborhood $\mathcal{O}_{\tilde{\theta}}$ of $\tilde{\theta}$ in $C^{0,\alpha}(\partial \Omega)$, and a real analytic   $\Theta$ from $\mathopen]-\epsilon',\epsilon'\mathclose[$ to $\mathcal{O}_{\tilde{\theta}}$ such that {the} set of zeros of $\Lambda$ in $\mathopen]-\epsilon',\epsilon'\mathclose[ \times \mathcal{O}_{\tilde{\theta}}$ coincides with the graph of $\Theta$. In particular,  $\Theta[0]=\tilde{\theta}$ and  $\Theta[\epsilon]=\theta_\epsilon$ for all $\epsilon \in \mathopen]0,\epsilon'\mathclose[$.
\end{proof}

We can now combine the representation formula for the solution $u_\epsilon$ as a single-layer potential with density $\epsilon^{-1}\Theta[\epsilon]$ and the regularity result of the map $\epsilon \mapsto \Theta[\epsilon]$ obtained in Proposition \ref{prop:Theta} to prove the following theorem on the real analytic continuation of the map $\epsilon \mapsto {u_\epsilon}$ for negative values of $\epsilon$.

\begin{theorem}\label{thm:rep}
Let $\alpha \in \mathopen]0,1[$ and $\Omega$ be as in assumption \eqref{Omega_def}. Let $p \in Q$. Let $\epsilon_0$ be as in assumption \eqref{epsilon_0}. Let $\eta \in \mathbb{R}^3$ and let {$k \in \mathbb{C}$ be such} that $k^2 \notin \left\{|2\pi q^{-1}z+\eta|^2 : z \in \mathbb{Z}^3\right\}$.     Let $g \in C^{1,\alpha}(\partial\Omega)$. Let $\epsilon^\#_0 $ be as in Proposition \ref{prop:bij}. Let $\tilde{\theta}$ be the unique solution in $C^{0,\alpha}(\partial \Omega)$ of equation \eqref{eq:lim}.  Let $\epsilon' \in \mathopen]0,\epsilon^\#_0[$ be as in Proposition \ref{prop:Theta}. Let $V$ be a bounded open subset of $\mathbb{R}^3$ such that $\overline{V} \cap (p+q\mathbb{Z}^3)=\emptyset$. Then, there exist $\epsilon''  \in \mathopen]0,\epsilon']$ and  a real analytic map $U$ from  $\mathopen]-\epsilon'',\epsilon''\mathclose[$ to $C^2(\overline{V})$ such that  
\begin{equation}\label{eq:rep1}
\overline{V} \subseteq \mathbb{S}_q[\Omega_{p,\epsilon}]^- \qquad \forall \epsilon \in \mathopen]-\epsilon'',\epsilon''[
\end{equation}
and 
\begin{equation}\label{eq:rep2}
u_\epsilon(x)=\epsilon U[\epsilon](x) \qquad \forall x \in \overline{V}\, 
\end{equation}
for all $\epsilon \in \mathopen]0,\epsilon''[$. Moreover,
\begin{equation}\label{eq:rep3}
U[0]=  {G^{k}_{q,\eta}}(x-p) \int_{\partial \Omega} \tilde{\theta}(s)\,d\sigma_s \qquad \forall x \in \overline{V}\, .
\end{equation}
\end{theorem}
\begin{proof}
Taking $\epsilon''  \in \mathopen]0,\epsilon']$ small enough, we can clearly assume that  condition \eqref{eq:rep1} holds. Let $U$ be the map from $\mathopen]-\epsilon'',\epsilon''[$ to  {$C^2(\overline{V})$} defined by
\[
U[\epsilon](x):= \int_{\partial \Omega}{G^{k}_{q,\eta}}(x-p-\epsilon s)\Theta[\epsilon](s)\,d\sigma_s  \qquad \forall x \in \overline{V},
\]
for all $\epsilon \in \mathopen]-\epsilon'',\epsilon''[$. By Proposition \ref{prop:Theta} and by \cite[Prop.~5.5]{BrDaLuMu24},  $U$ is real analytic and we have
\[
u_\epsilon(x)={\mathcal{S}^{k,-}_{q,\eta}}[\partial\Omega_{p,\epsilon},\epsilon^{-1}\theta_\epsilon ((\cdot-p)/\epsilon)](x)=\epsilon  U[\epsilon](x) \qquad \forall x \in \overline{V}\, ,
\]
for all $\epsilon \in \mathopen]0,\epsilon''[$. Then the validity of \eqref{eq:rep2} and \eqref{eq:rep3} follows immediately.
\end{proof}

Theorem \ref{thm:rep} implies, in particular, that if $\overline{x} \in V$, {then
\[
\epsilon \mapsto  \epsilon^{-1}u_\epsilon(\overline{x})
\]
admits a real analytic continuation on a whole neighborhood of $\epsilon=0$. Moreover, its Taylor series around $\epsilon=0$ coincides with that of
\[
\epsilon \mapsto U[\epsilon](\overline{x})\, .
\]
Since 
\[
U[0](\overline{x})={G^{k}_{q,\eta}}(\overline{x}-p) \int_{\partial \Omega} \tilde{\theta}(s)\,d\sigma_s
\]
and $\epsilon \mapsto U[\epsilon](\overline{x})$ is real analytic, we deduce that there exists a sequence $\{\tilde{a}_j\}_{j\in \mathbb{N}}$ of real numbers such that
\[
U[\epsilon](\overline{x})= {G^{k}_{q,\eta}}(\overline{x}-p) \int_{\partial \Omega} \tilde{\theta}(s)\,d\sigma_s+\sum_{j=1}^{+\infty}\tilde{a}_j\epsilon^{j} 
\]
for $\epsilon$ small enough. Moreover, the power series $\sum_{j=1}^{+\infty}\tilde{a}_j\epsilon^{j}$ converges for $\epsilon$ in a neighborhood of $0$. Therefore, if we set $a_j=\tilde{a}_{j-1}$ for all $j \in \mathbb{N}$ with $j \geq 2$, then
\[
u_\epsilon(\overline{x})=\epsilon {G^{k}_{q,\eta}}(\overline{x}-p) \int_{\partial \Omega} \tilde{\theta}(s)\,d\sigma_s +\sum_{j=2}^{+\infty}a_j \epsilon^j\, ,
\]
for $\epsilon$ positive and small enough, where $\sum_{j=2}^{+\infty}a_j \epsilon^j$ is a convergent power series for $\epsilon$ in a neighborhood of $0$. }{Note that if one is interested in knowing the subsequent coefficients {$\{a_j\}_{j \in \mathbb{N}}$} in the series expansion of $u_\epsilon(\overline{x})$, this can be done through a constructive procedure (see, for example, \cite{DaMuPu19}).}

\section{The case of systems: the Lam\'e equations}\label{s:lam}

The aim of this section is to show that the results of the previous sections on periodic (and quasi-periodic) layer potentials for elliptic equations can be extended to systems. As an example, we will consider the Lam\'e system.

{We set
\[
L[\omega]u{:=} \Delta u+\omega \nabla \mathrm{div}u \,,
\] 
for all regular functions}  {$u=(u_1,\dots,u_n)$ from a subset of $\mathbb{R}^n$  with values in $\mathbb{R}^n$. Here above, if $u=(u_1,\dots,u_n)$ is a sufficiently regular function with values in $\mathbb{R}^n$, the operator $\Delta$ is applied componentwise, {\it i.e.} $\Delta u = (\Delta u_1,\dots,\Delta u_n)$.}

In the following result of \cite[Thm.~3.1]{DaMu14} we introduce a periodic fundamental solution for the Lam\'e equations (see also Ammari and Kang \cite[Lemma 9.21]{AmKa07}). As we have already seen, this is the preliminary step to define periodic layer potentials for the Lam\'e equations.


\begin{theorem}
\label{psperL}
Let $\omega \in \mathopen]1-(2/n),+\infty[$. Let $\Gamma_{n,\omega}^{q}:= (\Gamma_{n,\omega,j}^{q,k})_{(j,k)\in\{1,\dots,n\}^2}$ be the $n \times n$ matrix of tempered distributions
 with $(j,k)$ entry defined by the generalized series
\[
\Gamma_{n,\omega,j}^{q,k}:= \sum_{z \in \mathbb{Z}^n \setminus \{0\}} \frac{1}{4 \pi^2 |Q|  |q^{-1}z|^2}\Biggl[ -\delta_{j,k}+\frac{\omega}{\omega+1}\frac{(q^{-1}z)_j(q^{-1}z)_k}{|q^{-1}z|^2}\Biggr]E_{2 \pi iq^{-1}z} \qquad \forall (j,k) \in \{1,\dots,n\}^2\,.
\]
{Here above, the symbol $\delta_{j,k}$ is the Kronecker delta function and equals $1$ if $j=k$ and $0$ if $j\neq k$.} Then $\Gamma_{n,\omega}^{q}$ is {$Q$-periodic} and is such that
\[
L[\omega] \Gamma_{n,\omega}^{q}=\sum_{z \in \mathbb{Z}^n}\delta_{qz}I_n-\frac{1}{ |Q|_n}I_n \, .
\]
Moreover, $\Gamma_{n,\omega}^{q}$ is real analytic from $\mathbb{R}^n \setminus q\mathbb{Z}^n$ to $M_n(\mathbb{R})$, $\Gamma_{n,\omega,j}^{q,k} \in L^1_{\mathrm{loc}}(\mathbb{R}^n)$, for all $(j,k)\in \{1,\dots,n\}^2$, and $\Gamma_{n,\omega}^q(x)=\Gamma_{n,\omega}^{q}(-x)$ for all $x \in \mathbb{R}^n \setminus q\mathbb{Z}^n$.
\end{theorem}

{In the theorem, $M_n(\mathbb{R})$ denotes the space of $n\times n$ matrices with real entries and $I_n$ denotes the $n\times n$ identity matrix.}

We find it convenient to set
\[
\Gamma_{n,\omega}^{q,j}:= \bigl(\Gamma_{n,\omega,i}^{q,j}\bigr)_{i \in \{1,\dots,n\}}\,,
\]
which we think as column vectors for all $j\in\{1,\dots,n\}$.  

{In order to define the layer potentials, we denote by $T$ the function from $ \mathopen]1-(2/n),+\infty\mathclose[\times M_n(\mathbb{R})$ to $M_n(\mathbb{R})$ defined by 
\[
T(\omega,A){:=} (\omega-1)(\mathrm{tr}A)I_n+(A+A^t)
\]  
for all  $\omega \in \mathopen]1-(2/n),+\infty\mathclose[$ and $A \in M_n(\mathbb{R})$, where $\mathrm{tr}A$ and $A^t$ denote the trace and the transpose matrix of $A$, respectively. }

We introduce the periodic double-layer potential. 
Let $\omega \in \mathopen]1-(2/n),+\infty[$. Let $\alpha \in \mathopen]0,1[$ and $\Omega_Q$ be as in \eqref{OmegaQ_def}. For  $\mu \in C^{0}(\partial {\Omega_Q},\mathbb{R}^n)$, we define the periodic double-layer potential $\mathcal{D}_q^{L,\omega}[\partial\Omega_Q, \mu]$ by setting
\[
\mathcal{D}_q^{L,\omega}[\partial\Omega_Q, \mu](x):= -\Biggl(\int_{\partial {\Omega_Q}}\mu^t(y)T(\omega, D \Gamma_{n,\omega}^{q,i}(x-y))\nu_{{\Omega_Q}}(y)\,d\sigma_y\Biggr)_{i\in\{1,\dots,n\}} \  \forall x \in \mathbb{R}^n\,,
\] which we think as a column vector. {Here above, $D \Gamma_{n,\omega}^{q,i}$ denotes the Jacobian matrix of $\Gamma_{n,\omega}^{q,i}$ for $i \in \{1,\dots,n\}$.}
Moreover, we set
\[
\mathcal{K}_q^{L,\omega}[\partial\Omega_Q, \mu]:=\mathcal{D}_q^{L,\omega}[\partial\Omega_Q, \mu]_{|\partial \Omega_Q}\qquad \text{on $\partial \Omega_Q$}\, .
\]

Then the standard properties of periodic layer potentials also hold for the periodic double-layer potential $\mathcal{D}_q^{L,\omega}[\partial\Omega_Q, \mu]$ (see \cite[Thm.~3.3]{DaMu14}).

\begin{theorem}
\label{dperpot}
Let $\omega \in \mathopen]1-(2/n),+\infty[$. Let $\alpha \in \mathopen]0,1[$ and $\Omega_Q$ be as in \eqref{OmegaQ_def}.  Then the following statements hold.
\begin{itemize}
\item[(i)] If $\mu\in {C^{0}(\partial{\Omega_Q},\mathbb{R}^n)}$, then $\mathcal{D}_q^{L,\omega}[\partial\Omega_Q,\mu]$ is {$Q$-periodic, $\mathcal{D}_q^{L,\omega}[\partial\Omega_Q,\mu]$ is of class $C^\infty(\mathbb{R}^n\setminus \partial \mathbb{S}_q[\Omega_Q],\mathbb{R}^n)$,} and
\[
L[\omega]\mathcal{D}_q^{L,\omega}[\partial\Omega_Q,\mu](x)=0\qquad\forall x\in {\mathbb{R}}^{n}\setminus\partial{\mathbb{S}}_q[{\Omega_Q}]\,.
\]
\item[(ii)] If $\mu\in C^{1,\alpha}(\partial{\Omega_Q},\mathbb{R}^n)$, then 
the restriction $\mathcal{D}_q^{L,\omega}[\partial\Omega_Q,\mu]_{|{\mathbb{S}}_q[{\Omega_Q}]}$ can be extended to a function $\mathcal{D}_q^{L,\omega,+}[\partial \Omega_Q,\mu] \in C_{q}^{1,\alpha}(\overline{{\mathbb{S}}_q[{\Omega_Q}]},\mathbb{R}^n)$ and
\[
\mathcal{D}_q^{L,\omega,+}[\partial \Omega_Q,\mu]=\frac{1}{2}\mu+
\mathcal{K}_q^{L,\omega}[\partial\Omega_Q,\mu]
\qquad{\mathrm{on}}\ \partial {\Omega_Q} \,.
\]
Moreover, the operator from  $C^{1,\alpha}(\partial{\Omega_Q},\mathbb{R}^n)$ to 
$C^{1,\alpha}_{q}(\overline{{\mathbb{S}}_q[{\Omega_Q}]},\mathbb{R}^n)$ that takes $\mu$ to $
\mathcal{D}_q^{L,\omega,+}[\partial \Omega_Q,\mu]$ is  linear and continuous. 
\item[(iii)] If $\mu\in C^{1,\alpha}(\partial{\Omega_Q},\mathbb{R}^n)$, then  the  restriction $\mathcal{D}_q^{L,\omega}[\partial\Omega_Q,\mu]_{|{\mathbb{S}}_q[{\Omega_Q}]^{-}}$ can be extended to a function $\mathcal{D}_q^{L,\omega,-}[\partial \Omega_Q,\mu] \in C^{1,\alpha}_{q}(\overline{{\mathbb{S}}_q[{\Omega_Q}]^{-}},\mathbb{R}^n)$ and 
\[
\mathcal{D}_q^{L,\omega,-}[\partial \Omega_Q,\mu]=-\frac{1}{2}\mu+
\mathcal{K}_q^{L,\omega}[\partial\Omega_Q,\mu]
\qquad{\mathrm{on}}\ \partial {\Omega_Q} \,.
\]
Moreover, the operator   from   $C^{1,\alpha}(\partial{\Omega_Q},\mathbb{R}^n)$ to $C^{1,\alpha}_{q}(\overline{{\mathbb{S}}_q[{\Omega_Q}]^{-}},\mathbb{R}^n)$ that takes $\mu$ to $\mathcal{D}_q^{L,\omega,-}[\partial \Omega_Q,\mu]$ is continuous. 
\item[(iv)] The map   from   $C^{1,\alpha}(\partial{\Omega_Q},\mathbb{R}^n)$ to itself that takes $\mu$ to $\mathcal{K}_q^{L,\omega}[\partial\Omega_Q,\mu]$ is continuous. 
\end{itemize}
\end{theorem}

Let $\alpha \in \mathopen]0,1[$ and $\Omega_Q$ be as in \eqref{OmegaQ_def}. Similarly, for  $\mu \in C^{0}(\partial {\Omega_Q},\mathbb{R}^n)$,  we denote by $\mathcal{S}_q^{L,\omega}[\partial \Omega_Q,\mu]$ the periodic single-layer potential, i.e., the function from $\mathbb{R}^n$ to $\mathbb{R}^n$ defined by setting
\[
\mathcal{S}_q^{L,\omega}[\partial \Omega_Q,\mu](x):= \int_{\partial {\Omega_Q}}\Gamma^q_{n,\omega}(x-y)\mu(y)\,d\sigma_y \qquad \forall x \in \mathbb{R}^n\,.
\]
 We also set
\[
(\mathcal{K}_q^{L,\omega})^\ast[\partial \Omega_Q,\mu](x):= \int_{\partial {\Omega_Q}}\sum_{l=1}^n \mu_{l}(y)T(\omega,D\Gamma_{n,\omega}^{q,l}(x-y))\nu_{{\Omega_Q}}(x)\,d\sigma_y \quad \forall x \in \partial {\Omega_Q}\,.
\]

Clearly, we can deduce the validity of similar properties also for the  periodic single-layer potential $\mathcal{S}_q^{L,\omega}[\partial\Omega_Q, \mu]$ (see \cite[Thm.~3.2]{DaMu14}).

\begin{theorem}
\label{sperpot}
Let $\omega \in \mathopen]1-(2/n),+\infty[$. Let $\alpha \in \mathopen]0,1[$ and $\Omega_Q$ be as in \eqref{OmegaQ_def}.  Then the following statements hold.
\begin{itemize}
\item[(i)] If $\mu\in {C^{0}(\partial{\Omega_Q},\mathbb{R}^n)}$, then $\mathcal{S}_q^{L,\omega}[\partial \Omega_Q,\mu]$ is {$Q$-periodic, $\mathcal{S}_q^{L,\omega}[\partial\Omega_Q,\mu]$ is of class $C^\infty(\mathbb{R}^n\setminus \partial \mathbb{S}_q[\Omega_Q],\mathbb{R}^n)$,} and
\[
L[\omega]\mathcal{S}_q^{L,\omega}[\partial \Omega_Q,\mu](x)
=
-\frac{1}{|Q|_n}\int_{\partial{\Omega_Q}}\mu \,d\sigma
\]
for all  $x\in 
 {\mathbb{R}}^{n}\setminus\partial{\mathbb{S}}_q[{\Omega_Q}]$.
\item[(ii)]  If $\mu\in C^{0,\alpha}(\partial{\Omega_Q},\mathbb{R}^n)$, then the function 
$\mathcal{S}_q^{L,\omega,+}[\partial \Omega_Q,\mu]:= \mathcal{S}_q^{L,\omega}[\partial \Omega_Q,\mu]_{|\overline{{\mathbb{S}}_q[{\Omega_Q}]}}$ belongs to $C^{1,\alpha}_{q}(\overline{{\mathbb{S}}_q[{\Omega_Q}]},\mathbb{R}^n)$  and  
\[
T\bigl(\omega,D\mathcal{S}_q^{L,\omega,+}[\partial \Omega_Q,\mu](x)\bigr)\nu_{{\Omega_Q}}(x)=-\frac{1}{2}\mu(x)+(\mathcal{K}_q^{L,\omega})^\ast[\partial\Omega_Q,\mu](x) \qquad \forall x \in \partial{\Omega_Q}\, ,
\]
for all $\mu \in C^{0,\alpha}(\partial{\Omega_Q},\mathbb{R}^n)$. Moreover, the operator
 that takes $\mu$ to 
$\mathcal{S}_q^{L,\omega,+}[\partial \Omega_Q,\mu]  $ is  continuous from $C^{0,\alpha}(\partial{\Omega_Q},\mathbb{R}^n)$ to $C^{1,\alpha}_{q}(\overline{{\mathbb{S}}_q[{\Omega_Q}]},\mathbb{R}^n)$. 
\item[(iii)]  If $\mu\in C^{0,\alpha}(\partial{\Omega_Q},\mathbb{R}^n)$, then the function 
$\mathcal{S}_q^{L,\omega,-}[\partial \Omega_Q,\mu]:= \mathcal{S}_q^{L,\omega}[\partial \Omega_Q,\mu]_{|\overline{{\mathbb{S}}_q[{\Omega_Q}]^{-}}}$ belongs to $C^{1,\alpha}_{q}
(\overline{{\mathbb{S}}_q[{\Omega_Q}]^{-}},\mathbb{R}^n)$  and we have
\[
T\bigl(\omega,D\mathcal{S}_q^{L,\omega,-}[\partial \Omega_Q,\mu](x)\bigr)\nu_{{\Omega_Q}}(x)=\frac{1}{2}\mu(x)+(\mathcal{K}_q^{L,\omega})^\ast[\partial\Omega_Q,\mu](x) \qquad \forall x \in \partial{\Omega_Q}\, ,
\]
for all $\mu \in C^{0,\alpha}(\partial{\Omega_Q},\mathbb{R}^n)$.
Moreover, the operator that takes $\mu$ to $\mathcal{S}_q^{L,\omega,-}[\partial \Omega_Q,\mu]$   is  continuous from  $ C^{0,\alpha}(\partial{\Omega_Q},\mathbb{R}^n)$ to $C^{1,\alpha}_{q}
(\overline{{\mathbb{S}}_q[{\Omega_Q}]^{-}},\mathbb{R}^n)$.
\item[(iv)]  The operator that takes $\mu$ to $(\mathcal{K}_q^{L,\omega})^\ast[\partial\Omega_Q,\mu]$ is  continuous from the space $C^{0,\alpha}(\partial{\Omega_Q},\mathbb{R}^n)$ to itself.
\end{itemize}
\end{theorem}

We observe that periodic layer potentials for the Lam\'e system have been used to analyze perturbation problems in periodic domains, for example, in \cite{DaMiMu23, DaMu14, FaLuMu21}.

\section{The extension to parabolic equations: space-periodic layer potentials for the heat equation}\label{s:heat}

In this section, we aim to show that the periodic potential theory approach is not limited to elliptic equations and systems, but it can also be applied in a non-elliptic setting, such as for parabolic equations. Specifically, we consider the heat equation
\[
\partial_t u-\Delta u=0
\]
and present a periodic fundamental solution for it (for a proof see \cite[Thm.~1]{Lu20}).

\begin{theorem} \label{Phinqthm}
Let   the function $\Phi_{q,n}$   
from $(\mathbb{R}\times {\mathbb{R}}^{n})\setminus (\{0\} \times q \mathbb{Z}^n)$ to ${\mathbb{R}}$ be
defined by
\begin{equation} \label{phinq}
\Phi_{q,n}(t,x):=
\left\{
\begin{array}{ll} 
\sum_{z\in \mathbb{Z}^n}\frac{1}{(4\pi t)^{\frac{n}{2}} }e^{-\frac{|x+qz|^{2}}{4t}}&{\mathrm{if}}\ (t,x)\in \mathopen]0,+\infty\mathclose[\times{\mathbb{R}}^{n}\,, 
 \\
 0 &{\mathrm{if}}\ (t,x)\in (\mathopen]-\infty,0]\times{\mathbb{R}}^{n})\setminus (\{0\} \times q \mathbb{Z}^n)\,.
\end{array}
\right.
\end{equation}
 Then the following statements hold.
\begin{itemize}
\item[(i)] The generalized series in (\ref{phinq}) which defines $\Phi_{q,n}$ converges uniformly on the compact subsets of 
$\mathopen]0,+\infty\mathclose[ \times \mathbb{R}^n$.
\item[(ii)] Let $K$ be a compact subset of $\mathbb{R}^n$ such that $K \cap q\mathbb{Z}^n = \varnothing$. Then 
\[
\lim_{t \to 0^+} \Phi_{q,n}(t, x) = 0,
\]
uniformly with respect to $x \in K$.
\item[(iii)]$\Phi_{q,n}$ is {$Q$-periodic} and $\Phi_{q,n} \in C^{\infty}\left(({\mathbb{R}}\times {\mathbb{R}}^{n})\setminus (\{0\} \times q \mathbb{Z}^n)\right)$.
	Moreover, $\Phi_{q,n}$ solves the heat equation in $({\mathbb{R}}\times {\mathbb{R}}^{n})\setminus (\{0\} \times q \mathbb{Z}^n)$.
\item[(iv)] Let $f \in C^0(\mathbb{R}^n)$ such that $f$ is {$Q$-periodic}. Let $u$ be the function from $]0,+\infty[ \times \mathbb{R}^n$  to $\mathbb{C}$ 
defined by 
\[
u(t,x) := \int_Q \Phi_{q,n}(t,x-y)f(y)\,dy \qquad \forall\,(t,x) \in \mathopen]0,+\infty\mathclose[ \times \mathbb{R}^n.
\]
Then $u$ belongs to $C^\infty(\mathopen]0,+\infty\mathclose[ \times \mathbb{R}^n)$, solves the heat equation in $]0,+\infty[ \times \mathbb{R}^n$, 
 is {$Q$-periodic} and 
\[
\lim_{t \to 0^+} u(t,x)=f(x) \qquad \forall\,x \in  \mathbb{R}^n.
\]
\item[(v)] 
\[
\Phi_{q,n}(t,x)  = \sum_{z\in \mathbb{Z}^n}\frac{1}{|Q|_n} e^{-4\pi^2|q^{-1}z|^2t + 2\pi i (q^{-1}z) \cdot x} \qquad \forall\,(t,x) \in \mathopen]0, +\infty\mathclose[ \times \mathbb{R}^n\, .
\]
\end{itemize}
\end{theorem}

If $\tilde{\Omega}$ is a subset of $\mathbb{R}^n$, we  set
\[
\tilde{\Omega}_{T}:= [0,T]\times \tilde{\Omega}\,,
\qquad
\partial_{T}\tilde{\Omega}:= (\partial \tilde{\Omega})_{T}=[0,T]\times\partial \tilde{\Omega}\,.
\]

Then we are in the position to introduce the {$Q$-periodic} in space  layer heat potentials.    Let $\alpha \in \mathopen]0,1[$ and $\Omega_Q$ be as in \eqref{OmegaQ_def}.  For a   density (or moment) $\mu$ 
in {$C^0(\partial_T\Omega_Q)$}, we set

\[
\mathcal{S}_q^{\mathrm{h}}[\partial_T\Omega_Q,\mu](t,x) 
 := \int_{0}^{t} \int_{\partial \Omega_Q} \Phi_{q,n}(t-\tau,x-y)  \mu(\tau, y)\,d\sigma_y d\tau \qquad \forall\,(t,x) \in [0,T] \times \mathbb{R}^n,
\]
 
\[
\mathcal{D}_q^{\mathrm{h}}[\partial_T\Omega_Q,\mu](t,x) := - \int_{0}^t\int_{\partial\Omega_Q}\nu_{\Omega_Q}(y) \cdot \nabla_x  
\Phi_{q,n}(t-\tau,x-y)\mu(\tau,y)\,d\sigma_yd\tau \qquad \forall\,(t,x) \in [0,T] \times \mathbb{R}^n,
\]
  
\[
\mathcal{K}_{q}^{\mathrm{h}}[\partial_T\Omega_Q,\mu](t,x) := - \int_{0}^t\int_{\partial\Omega_Q}\nu_{\Omega_Q}(y) \cdot \nabla_x
\Phi_{q,n}(t-\tau,x-y)\mu(\tau,y)\,d\sigma_yd\tau \qquad \forall\,(t,x) \in  \partial_T\Omega_Q,
\]

\[
\mathcal{K}_{q,\ast}^{\mathrm{h}}[\partial_T\Omega_Q,\mu](t,x) :=  \int_{0}^t\int_{\partial\Omega_Q}\nu_{\Omega_Q}(x) \cdot \nabla_x
\Phi_{q,n}(t-\tau,x-y)\mu(\tau,y)\,d\sigma_yd\tau \qquad \forall\,(t,x) \in  \partial_T\Omega_Q.
\]

The function $\mathcal{S}_q^{\mathrm{h}}[\partial_T\Omega_Q,\mu]$ is the {$Q$-periodic} in space single-layer heat potential with density $\mu$ and the function $\mathcal{D}_q^{\mathrm{h}}[\partial_T\Omega_Q,\mu]$ is the {$Q$-periodic} in space double-layer heat potential with density $\mu$.

{
\begin{remark}
Note that in a more general definition of parabolic layer heat potentials the integration with respect to the time variable $t$ is done on the unbounded interval $\mathopen]-\infty,t\mathclose[$ instead of $]0,t[$ (see \cite{LaLu17,LaLu19}). However, when considering initial-boundary value problem for heat equation with homogeneous initial condition,  the above definition is enough.
\end{remark}
}

As in the elliptic case, our functions are also framed within the context of Schauder spaces. We refer to Lady\v{z}enskaja,  Solonnikov, and Ural'ceva \cite{LaSoUr68} for the definition of 
time-dependent functions  in the parabolic Schauder class $C^{\frac{j+\alpha}{2};j+\alpha}$ on $[0,T]\times \overline{\mathcal{O}}$ or $[0,T]\times \partial\mathcal{O}$, where $\mathcal{O}$ is an open set {and $j \in \{0,1\}$}.  In other words, a function of class $C^{\frac{j+\alpha}{2};j+\alpha}$ is 
$\left(\frac{j+\alpha}{2}\right)$-H\"older continuous in the time variable, and $(j,\alpha)$-Schauder regular in the space variable. We also denote by  $C_0^{\frac{j+\alpha}{2};j+\alpha}$ the parabolic Schauder class of functions vanishing at time $t=0$, and by $C_{0,q}^{\frac{j+\alpha}{2};j+\alpha}$ the subspace of $C_0^{\frac{j+\alpha}{2};j+\alpha}$ consisting of functions {$f$ such that $f(t,\cdot)$ is  $Q$-periodic for every $t \in [0,T]$ (in this case, for brevity, we say that $f$ is $Q$-periodic in space).} Clearly, we can extend the definition of parabolic Schauder classes to products of intervals and manifolds by using local charts.

  In the following theorem of \cite[Thm.~3]{Lu20} we collect some properties of the periodic double-layer potential for the heat equation.

\begin{theorem} \label{thmdl}
Let $\alpha \in \mathopen]0,1[$, $T>0$. Let $\Omega_Q$ be as in \eqref{OmegaQ_def}. Then the following statements hold.
\begin{itemize}
\item[(i)] If $\mu \in {C^0(\partial_T\Omega_Q)}$, then $\mathcal{D}_q^{\mathrm{h}}[\partial_T\Omega_Q, \mu]$  is {$Q$-periodic} in space, 
$\mathcal{D}_q^{\mathrm{h}}[\partial_T\Omega_Q,\mu] \in  C^\infty(]0,T]\times (\mathbb{R}^n \setminus \partial\mathbb{S}_q[\Omega_Q]))$, and $\mathcal{D}_q^{\mathrm{h}}[\partial_T\Omega_Q,\mu]$ solves the heat equation in 
$]0,T]\times (\mathbb{R}^n \setminus \partial\mathbb{S}_q[\Omega_Q])$.
\item[(ii)]   If $\mu \in C^{\frac{1+\alpha}{2};1+\alpha}_0(\partial_T \Omega_Q)$,  then the restriction $\mathcal{D}_q^{\mathrm{h}}[\partial_T\Omega_Q, \mu]_{|\mathbb{S}_q[\Omega_Q]_T}$ can be 
extended uniquely to an element $\mathcal{D}_q^{\mathrm{h},+}[\partial_T\Omega_Q, \mu] \in C_{0,q}^{\frac{1+\alpha}{2};1+\alpha}(\overline{\mathbb{S}_q[\Omega_Q]_T})$ and 
\[
\mathcal{D}_q^{\mathrm{h},+}[\partial_T\Omega_Q, \mu](t,x) = - \frac{1}{2} \mu(t,x) + \mathcal{K}_q^{\mathrm{h}}[\partial_T\Omega_Q, \mu](t,x) \qquad {\forall (t,x) \in \partial_T\Omega_Q\, .}
\]
Moreover, the operator from $C^{\frac{1+\alpha}{2};1+\alpha}_0(\partial_T \Omega_Q)$ to $C_{0,q}^{\frac{1+\alpha}{2};1+\alpha}(\overline{\mathbb{S}_q[\Omega_Q]_T})$, 
that takes $\mu$ to the function $\mathcal{D}_q^{\mathrm{h},+}[\partial_T\Omega_Q, \mu]$, is linear and continuous. 
\item[(iii)]   If $\mu \in C^{\frac{1+\alpha}{2};1+\alpha}_0(\partial_T \Omega_Q)$,  then the restriction $\mathcal{D}_q^{\mathrm{h}}[\partial_T\Omega_Q, \mu]_{|\mathbb{S}_q[\Omega_Q]^-_T}$ can be 
extended uniquely to an element $\mathcal{D}_q^{\mathrm{h},-}[\partial_T\Omega_Q, \mu] \in C_{0,q}^{\frac{1+\alpha}{2};1+\alpha}(\overline{\mathbb{S}_q[\Omega_Q]^-_T})$ and 
\[
\mathcal{D}_q^{\mathrm{h},-}[\partial_T\Omega_Q, \mu](t,x) = \frac{1}{2} \mu(t,x) + \mathcal{K}_q^{\mathrm{h}}[\partial_T\Omega_Q, \mu](t,x),
\]
for all $(t,x) \in \partial_T\Omega_Q$.
Moreover, the operator from $C^{\frac{1+\alpha}{2};1+\alpha}_0(\partial_T \Omega_Q)$ to $C_{0,q}^{\frac{1+\alpha}{2};1+\alpha}(\overline{\mathbb{S}_q[\Omega_Q]^-_T})$, 
that takes $\mu$ to the function $\mathcal{D}_q^{\mathrm{h},-}[\partial_T\Omega_Q, \mu]$, is linear and continuous. 
\end{itemize}
\end{theorem}

Concerning the periodic single-layer potential, we have instead the following  (see \cite[Thm.~2]{Lu20}).

\begin{theorem}\label{thmsl}
Let $\alpha \in \mathopen]0,1[$, $T>0$. Let $\Omega_Q$ be as in \eqref{OmegaQ_def}. Then the following statements hold.
\begin{itemize}

\item[(i)]  If $\mu \in {C^0(\partial_T\Omega_Q)}$, then $\mathcal{S}_{q}^{\mathrm{h}}[\partial_T\Omega_Q,\mu]$  is continuous, {$Q$-periodic} in space, 
$\mathcal{S}_{q}^{\mathrm{h}}[\partial_T\Omega_Q,\mu] \in C^\infty(]0,T] \times (\mathbb{R}^n \setminus \partial\mathbb{S}_q[\Omega_Q]))$ 
and $\mathcal{S}_{q}^{\mathrm{h}}[\partial_T\Omega_Q,\mu]$ solves the heat equation 
in $]0,T]\times (\mathbb{R}^n \setminus \partial\mathbb{S}_q[\Omega_Q])$. 
%
%
\item[(ii)] If $\mu \in C_0^{\frac{\alpha}{2}; \alpha}(\partial_T\Omega_Q)$ and we denote  by $\mathcal{S}_{q}^{\mathrm{h},+}[\partial_T\Omega_Q,\mu]$  
the restriction of $\mathcal{S}_{q}^{\mathrm{h}}[\partial_T\Omega_Q,\mu]$ to $\overline{\mathbb{S}_q[\Omega_Q]_T}$, then $\mathcal{S}_{q}^{\mathrm{h},+}[\partial_T\Omega_Q,\mu]\in C_{0,q}^{\frac{1+\alpha}{2}; 1+\alpha}(\overline{\mathbb{S}_q[\Omega_Q]_T})$ and  the map from  $C_0^{\frac{\alpha}{2}; \alpha}(\partial_T\Omega_Q)$ to  $C_{0,q}^{\frac{1+\alpha}{2}; 1+\alpha}(\overline{\mathbb{S}_q[\Omega_Q]_T})$ that takes $\mu$ to $\mathcal{S}_{q}^{\mathrm{h},+}[\partial_T\Omega_Q,\mu]$ is linear and continuous. Moreover, 
\[
\nu_{\Omega_Q}(x)\cdot \nabla_x \mathcal{S}_q^{\mathrm{h},+}[\partial_T\Omega_Q, \mu](t,x) =  \frac{1}{2}\mu(t,x)
 +\mathcal{K}_{q,\ast}^{\mathrm{h}}[\partial_T\Omega_Q,\mu](t,x) \qquad {\forall (t,x) \in \partial_T\Omega_Q\, .}
 \]

\item[(iii)] If $\mu \in C_0^{\frac{\alpha}{2}; \alpha}(\partial_T\Omega_Q)$ and we denote  by $\mathcal{S}_{q}^{\mathrm{h},-}[\partial_T\Omega_Q,\mu]$  
the restriction of $\mathcal{S}_{q}^{\mathrm{h}}[\partial_T\Omega_Q,\mu]$ to $\overline{\mathbb{S}_q[\Omega_Q]^-_T}$, then $\mathcal{S}_{q}^{\mathrm{h},-}[\partial_T\Omega_Q,\mu]\in C_{0,q}^{\frac{1+\alpha}{2}; 1+\alpha}(\overline{\mathbb{S}_q[\Omega_Q]^-_T})$ and  the map from  $C_0^{\frac{\alpha}{2}; \alpha}(\partial_T\Omega_Q)$ to  $C_{0,q}^{\frac{1+\alpha}{2}; 1+\alpha}(\overline{\mathbb{S}_q[\Omega_Q]^-_T})$, that takes $\mu$ to $\mathcal{S}_{q}^{\mathrm{h},-}[\partial_T\Omega_Q,\mu]$, is linear and continuous. Moreover, 
\[
\nu_{\Omega_Q}(x)\cdot \nabla_x\mathcal{S}_q^{\mathrm{h},-}[\partial_T\Omega_Q, \mu](t,x) =  - \frac{1}{2}\mu(t,x)
 +\mathcal{K}_{q,\ast}^{\mathrm{h}}[\partial_T\Omega_Q,\mu](t,x),
 \]
for all $(t,x) \in \partial_T\Omega_Q$.
%
%
%
\end{itemize}
\end{theorem}

\subsection{A regularly perturbed boundary value problem}\label{ss:regheat}

In this section, we want to show the application of periodic potential theory for the heat equation to the study of a regularly perturbed boundary value problem. More precisely, we will consider a Dirichlet problem for the heat equation in the product of a bounded interval with a periodically perforated domain, where the holes are deformed copies of a reference shape. We observe that, in \cite[\S 5]{DaLuMoMu24}, the periodic single-layer has been used to study a perturbed periodic transmission problem for the heat equation, whereas here we will study the Dirichlet problem using the periodic double-layer potential.

So, for $\alpha\in\mathopen]0,1[$, we assume that 
\begin{equation}\label{condOmega}
\begin{split}
&\Omega\text{ is a bounded connected open subset of } \mathbb{R}^n\  \text{of class } C^{1,\alpha}\\
& \text {and has connected exterior }\mathbb{R}^n\setminus\overline{\Omega}\,. 
\end{split}   
\end{equation}
The set $\Omega$ will play the role of a reference shape. In order to define regular domain perturbations, we consider specific classes of diffeomorphisms defined on the boundary $\partial \Omega$ of $\Omega$.

We  denote by $\mathcal{A}^{1,\alpha}_{\partial \Omega}$ the set of functions of class $C^{1,\alpha}(\partial\Omega, \mathbb{R}^{n})$ that are injective together with their differential at all points  of $\partial\Omega$. By Lanza de Cristoforis and Rossi \cite[Lem. 2.2, p. 197]{LaRo08} and \cite[Lem. 2.5, p. 143]{LaRo04}, we know that $\mathcal{A}^{1,\alpha}_{\partial \Omega}$ is an open subset of $C^{1,\alpha}(\partial\Omega, \mathbb{R}^{n})$.

For $\phi \in \mathcal{A}^{1,\alpha}_{\partial \Omega}$, by the Jordan-Leray separation theorem, $\mathbb{R}^{n}\setminus \phi(\partial \Omega)$ has exactly two open connected components {(see, e.g.,  \cite[\S A.4]{DaLaMu21}).} We denote the bounded connected component of $\mathbb{R}^{n}\setminus \phi(\partial \Omega)$ by $\mathbb{I}[\phi]$ and by $\nu_{\mathbb{I}[\phi]}$  the outer unit normal to $\mathbb{I}[\phi]$.
 
We also set 
\[
{\mathcal{A}}_{\partial\Omega,Q}^{1,\alpha} := \left\{\phi \in\mathcal{A}_{\partial \Omega}^{1,\alpha} : \phi(\partial\Omega) \subseteq Q\right\}, 
\]
and we define
\[
\mathbb{S}_q[\phi]^-:= \mathbb{S}_q[\mathbb{I}[\phi]]^-
\]
for all $\phi \in {\mathcal{A}}_{\partial\Omega,Q}^{1,\alpha}$.


%

We now introduce the perturbed Dirichlet problem we wish to study. Before doing so, we need to introduce some notation: if $D$ is a subset of $\mathbb{R}^n$,
$T >0$ and $h$ is a map from $D$ to $\mathbb{R}^n$, we denote by $h^T$ the map from  $[0,T] \times D$
 to  $[0,T] \times \mathbb{R}^n$ defined by 
\[
h^T(t,x) := (t, h(x)) \qquad \forall (t,x) \in  [0,T] \times D.
\]
Now, let $\alpha \in \mathopen]0,1[$, $T>0$, and $\Omega$ be as in assumption \eqref{condOmega}. Let  $\phi \in {\mathcal{A}}_{\partial\Omega,Q}^{1,\alpha}$.
Let  $f \in C^{\frac{1+\alpha}{2};1+\alpha}_0([0,T]\times\partial \Omega)$. {The $\phi$-dependent Dirichlet problem we consider is the following:}

\begin{align}\label{periodicdirheat}
\left\{
\begin{array}{ll}
\partial_t u - \Delta u = 0 &\mbox{in }  ]0, T]\times \mathbb{S}_q[\phi]^-, \\
u(t,x+qz) = u(t,x) &\forall\,(t,x) \in  [0,T] \times   \overline{\mathbb{S}_q[\phi]^-}, \, \forall\,z\in \mathbb{Z}^n, \\
u= f \circ (\phi^T)^{(-1)}&\mbox{on } [0, T]\times \partial \mathbb{I}[\phi],\\
u(0,\cdot) = 0 &\mbox{in }  \overline{ \mathbb{S}_q[\phi]^-}.
\end{array}
\right.
\end{align}

In the proposition below, we show that problem \eqref{periodicdirheat} has a unique solution and that it can be represented by a double-layer potential whose density solves a specific boundary integral equation.

\begin{proposition}\label{thm:solheat}
Let $\alpha\in\mathopen]0,1[$ and $T>0$. Let $\Omega$ be as in assumption \eqref{condOmega}.
Let  $\phi \in {\mathcal{A}}_{\partial\Omega,Q}^{1,\alpha}$.
Let  $f \in C^{\frac{1+\alpha}{2};1+\alpha}_0([0,T]\times\partial \Omega)$. Then problem \eqref{periodicdirheat} has
a unique solution
\[
u[\phi] \in C^{\frac{1+\alpha}{2};1+\alpha}_{0,q}([0,T]\times\overline{\mathbb{S}_q[\phi]^-}).
\]
Moreover,
\[
 u[\phi]=\mathcal{D}_q^{\mathrm{h},-}[\partial_T\mathbb{I}[\phi], \theta[\phi] \circ (\phi^T)^{(-1)}],
\]
where $\theta[\phi]$  is the unique solution in $C^{\frac{1+\alpha}{2};1+\alpha}_{0}([0,T] \times \partial\Omega)$ of the  integral equation
\begin{equation}\label{sys1}
 \frac{1}{2}\theta \circ (\phi^T)^{(-1)} + \mathcal{K}_q^{\mathrm{h}}[\partial_T\mathbb{I}[\phi], \theta \circ (\phi^T)^{(-1)}]=f \circ (\phi^T)^{(-1)}.
\end{equation}
 \end{proposition}
 \begin{proof}
 By \cite[Thm.~6]{Lu20}, we know that problem \eqref{periodicdirheat} has
a unique solution in $C^{\frac{1+\alpha}{2};1+\alpha}_{0,q}([0,T]\times\overline{\mathbb{S}_q[\phi]^-})$ and that such a solution can be represented as the double-layer potential $\mathcal{D}_q^{\mathrm{h},-}[\partial_T\mathbb{I}[\phi], \theta[\phi] \circ (\phi^T)^{(-1)}]$ where $\theta[\phi]$ is the unique solution in $C^{\frac{1+\alpha}{2};1+\alpha}_{0}([0,T] \times \partial\Omega)$ of  equation \eqref{sys1} (see also \cite[Thm.~3 and Lem.~4]{Lu20} and Theorem \ref{thmdl}).
 \end{proof}
 
 Our aim is to study the dependence of the `domain-to-solution' map $\phi \mapsto u[\phi]$. By Proposition \ref{thm:solheat}, problem \eqref{periodicdirheat} is equivalent to the integral equation \eqref{sys1}. As a consequence, a preliminary step is to study the regularity of the map $\phi \mapsto \theta[\phi]$, where $\theta[\phi]$ is the unique solution of \eqref{sys1}.
 We do so in the following proposition.
 
 
 \begin{proposition}\label{prop:repsolheat}
Let $\alpha\in\mathopen]0,1[$ and $T>0$. Let $\Omega$ be as in assumption \eqref{condOmega}.
Let  $f \in C^{\frac{1+\alpha}{2};1+\alpha}_0([0,T]\times\partial \Omega)$. Then the map from   ${\mathcal{A}}_{\partial\Omega,Q}^{1,\alpha}$ to $C^{\frac{1+\alpha}{2};1+\alpha}_{0}([0,T] \times \partial\Omega)$, that takes $\phi$ to the unique solution $\theta[\phi]$ in $C^{\frac{1+\alpha}{2};1+\alpha}_{0}([0,T] \times \partial\Omega)$ of the  integral equation
\eqref{sys1}, is of class $C^\infty$.
 \end{proposition}
 \begin{proof}
 We first note that equation \eqref{sys1} can be rewritten as
 \begin{equation}\label{eq:repsolheat1}
  \frac{1}{2}\theta  + \mathcal{K}_q^{\mathrm{h}}[\partial_T\mathbb{I}[\phi], \theta \circ (\phi^T)^{(-1)}]\circ (\phi^T)=f\, . 
 \end{equation}
 Then for each $\phi \in {\mathcal{A}}_{\partial\Omega,Q}^{1,\alpha}$, we introduce the bounded linear operator   
 \[
 \mathbf{K}_q^{\mathrm{h}}[\phi,\cdot] \in \mathcal{L}(C^{\frac{1+\alpha}{2};1+\alpha}_{0}([0,T] \times \partial\Omega), C^{\frac{1+\alpha}{2};1+\alpha}_{0}([0,T] \times \partial\Omega))
 \] 
 by setting
 \[
  \mathbf{K}_q^{\mathrm{h}}[\phi,\theta]:=\mathcal{K}_q^{\mathrm{h}}[\partial_T\mathbb{I}[\phi], \theta \circ (\phi^T)^{(-1)}]\circ (\phi^T)
 \]
 for all $\theta \in C^{\frac{1+\alpha}{2};1+\alpha}_{0}([0,T] \times \partial\Omega)$. As a consequence, we can rewrite equation \eqref{eq:repsolheat1} as 
 \[
   \frac{1}{2}\theta  +   \mathbf{K}_q^{\mathrm{h}}[\phi,\theta]=f\, . 
 \]
 By \cite[Lem.~4]{Lu20} and \cite[Thm.~4.3]{DaLuMoMu24}, we know that for each $\phi \in {\mathcal{A}}_{\partial\Omega,Q}^{1,\alpha}$ the linear and continuous operator  
 \[
    \frac{1}{2}\mathbf{I}  +   \mathbf{K}_q^{\mathrm{h}}[\phi,\cdot]
 \]
from $C^{\frac{1+\alpha}{2};1+\alpha}_{0}([0,T] \times \partial\Omega)$ to itself  is invertible and that the map from ${\mathcal{A}}_{\partial\Omega,Q}^{1,\alpha}$ to $ \mathcal{L}(C^{\frac{1+\alpha}{2};1+\alpha}_{0}([0,T] \times \partial\Omega), C^{\frac{1+\alpha}{2};1+\alpha}_{0}([0,T] \times \partial\Omega))$ that takes $\phi$ to $  \frac{1}{2}\mathbf{I}  +   \mathbf{K}_q^{\mathrm{h}}[\phi,\cdot]$ is of class $C^\infty$. Since the map that takes a linear invertible operator to its inverse is real analytic (cf.~Hille and Phillips \cite[Thms. 4.3.2 and 4.3.4]{HiPh57}), and therefore of class $C^\infty$, we also deduce that  the map from ${\mathcal{A}}_{\partial\Omega,Q}^{1,\alpha}$ to $ \mathcal{L}(C^{\frac{1+\alpha}{2};1+\alpha}_{0}([0,T] \times \partial\Omega), C^{\frac{1+\alpha}{2};1+\alpha}_{0}([0,T] \times \partial\Omega))$ that takes $\phi$ to 
\[
\bigg(\frac{1}{2}\mathbf{I}  +   \mathbf{K}_q^{\mathrm{h}}[\phi,\cdot]\bigg)^{(-1)}
\]
is of class $C^\infty$. Since for each $\phi \in {\mathcal{A}}_{\partial\Omega,Q}^{1,\alpha}$ we have
\[
\theta[\phi]=\bigg(\frac{1}{2}\mathbf{I}  +   \mathbf{K}_q^{\mathrm{h}}[\phi,\cdot]\bigg)^{(-1)} f\, ,
\]
we finally deduce that the map from   ${\mathcal{A}}_{\partial\Omega,Q}^{1,\alpha}$ to $C^{\frac{1+\alpha}{2};1+\alpha}_{0}([0,T] \times \partial\Omega)$, that takes $\phi$ to  $\theta[\phi]$,  is of class $C^\infty$.
 \end{proof}

To understand the dependence of the double-layer potential representing the solution upon $\phi$, we need the following technical lemma, which states that the map related to the change of variables in the area element and the pullback $\nu_{\mathbb{I}[\phi]}\circ\phi$ of the outer normal field depend analytically on $\phi$. For a  proof, we refer to Lanza de Cristoforis and Rossi \cite[p.~166]{LaRo04} and Lanza de Cristoforis \cite[Prop. 1]{La07}.

\begin{lemma}\label{rajacon}
Let $\alpha\in \mathopen ]0,1[$ and  $\Omega$ be as in assumption  \eqref{condOmega}.   Then the following statements hold.
\begin{itemize}
\item[(i)] For each $\phi \in \mathcal{A}^{1,\alpha}_{\partial \Omega}$, there exists a unique  
$\tilde \sigma_n[\phi] \in C^{0,\alpha}(\partial\Omega)$ such that $\tilde \sigma_n[\phi] > 0$ and 
\[ 
  \int_{\phi(\partial\Omega)}w(s)\,d\sigma_s=  \int_{\partial\Omega}w \circ \phi(y)\tilde\sigma_n[\phi](y)\,d\sigma_y, \qquad \forall w \in L^1(\phi(\partial\Omega)).
\]
Moreover, the map $\tilde \sigma_n[\cdot]$  is real analytic from $\mathcal{A}_{\partial \Omega}^{1,\alpha}  $ to $ C^{0,\alpha}(\partial\Omega)$.
\item[(ii)] The map from $\mathcal{A}_{\partial \Omega}^{1,\alpha} $ to $ C^{0,\alpha}(\partial\Omega, \mathbb{R}^{n})$, that takes $\phi$ to $\nu_{\mathbb{I}[\phi]} \circ \phi$, is real analytic.
\end{itemize}
\end{lemma}

Now, we can combine the representation of the solution $u[\phi]$ as a double-layer potential and the $C^\infty$-regularity of $\theta[\phi]$ upon the (infinite-dimensional) perturbation parameter $\phi$ to prove that suitable restrictions of $u[\phi]$ depend smoothly upon $\phi$. 

\begin{theorem}
Let $\alpha\in\mathopen]0,1[$ and $T>0$. Let $\Omega$ be as in assumption  \eqref{condOmega}. Let  $\Omega_0$ be a bounded open subset of $\mathbb{R}^n$. Let ${\mathcal{B}}_{\partial\Omega,Q}^{1,\alpha}$
be the open subset of ${\mathcal{A}}_{\partial\Omega,Q}^{1,\alpha}$ consisting of those diffeomorphisms $\phi$
such that
\[
\overline{\Omega_0} \subseteq \mathbb{S}_q[\phi]^-.
\]
 Then the map 
 \[
 \phi  \mapsto u[\phi]_{|[0,T]\times\overline{\Omega_0}}
 \]
 is of class $C^\infty$ from ${\mathcal{B}}_{\partial\Omega,Q}^{1,\alpha}$ to  $C_0^{\frac{1+\alpha}{2},1+\alpha}([0,T]\times \overline{\Omega_0})$.
\end{theorem}
 \begin{proof}
 We proceed as in the proof of \cite[Thm.~5.6]{DaLuMoMu24}. Without loss of generality we can assume that $\Omega_0$ is of class $C^{1,\alpha}$. If $\phi \in {\mathcal{B}}_{\partial\Omega,Q}^{1,\alpha}$, then
 \begin{equation}\label{repfu2}
    u[\phi](t,x)= 
- \int_{0}^t\int_{\partial\Omega}\nu_{\mathbb{I}[\phi]}\circ \phi (\eta) \cdot \nabla_x
\Phi_{q,n} (t-\tau, x - \phi(\eta)) \theta[\phi](\tau,\eta) \tilde \sigma_n[\phi](\eta)\, d\sigma_\eta d\tau,
\end{equation}
for all $(t,x) \in [0,T] \times   {\overline{\mathbb{S}_q[\phi]^-}}$. Then we observe that {the} map, that takes a diffeomorphism $\phi$ to the function
\[
\overline{\Omega_0} \times \partial \Omega \ni (x,\eta) \mapsto  x -\phi(\eta)\in \mathbb{R}^n,
\]
is  of class $C^\infty$ from ${\mathcal{B}}_{\partial\Omega,Q}^{1,\alpha}$ to $C^{1,\alpha}(\overline{\Omega_0} \times \partial \Omega, \mathbb{R}^n \setminus q\mathbb{Z}^n)$. 
 By arguing as in the proof of \cite[Lem. A.1 and Lem. A.3]{DaLu23} regarding the regularity of superposition operators, we deduce that the map, that takes $\phi$ to the function
\[
  \nabla_x
\Phi_{q,n} (s, x-\phi(\eta))\qquad\forall (s,x,\eta)\in  [0,T]\times\overline{\Omega_0} \times \partial \Omega,
\]
is of class $C^{\infty}$ from ${\mathcal{B}}_{\partial\Omega,Q}^{1,\alpha}$ to $C_0^{\frac{1+\alpha}{2};1+\alpha}([0,T]\times (\overline{\Omega_0} \times \partial \Omega))$.

Then, the statement follows by the representation formula  \eqref{repfu2} for $u[\phi]$, by Proposition \ref{prop:repsolheat} on the smoothness of $\theta[\phi]$, by Lemma \ref{rajacon} on the analyticity of $\tilde \sigma_n[\phi]$ and of $\nu_{\mathbb{I}[\phi]}\circ \phi$, and by the regularity result on integral operators with non-singular kernels of   \cite[Lem.~A.2]{DaLu23}. \end{proof}

\section*{Acknowledgment}

The authors are members of the ``Gruppo Nazionale per l'Analisi Matematica, la Probabilit\`a e le loro Applicazioni'' (GNAMPA) of the ``Istituto Nazionale di Alta Matematica'' (INdAM). . Dalla Riva, P. Luzzini, and P. Musolino acknowledge the support of the
project funded by the EuropeanUnion - NextGenerationEU under the National Recovery and
Resilience Plan (NRRP), Mission 4 Component 2 Investment 1.1 - Call PRIN 2022 No. 104 of
February 2, 2022 of Italian Ministry of University and Research; Project 2022SENJZ3 (subject area: PE - Physical Sciences and Engineering) ``Perturbation problems and asymptotics for elliptic differential equations: variational and potential theoretic methods.''  M. Dalla Riva also  acknowledges the support  by MUR (Ministero dell'Universit\`a e della Ricerca) through the PNRR Project QUANTIP -- Partenariato Esteso NQSTI -- PE00000023 -- Spoke 9 -- CUP: E63C22002180006.

\end{document}